\documentclass[11pt]{amsart}
\usepackage{amssymb}
\usepackage{latexsym}
\usepackage{graphicx}
\usepackage{subfigure}
\usepackage[linktocpage=true]{hyperref}
\usepackage{ulem}
\usepackage[left=3.8cm, right=3.8cm, top=3cm, bottom=3cm]{geometry}
\usepackage{color}
\usepackage{verbatim}
\usepackage{bm}
\usepackage{enumerate}
\usepackage{amsmath}
\usepackage{fancyhdr}
\hypersetup{ colorlinks   = true, 
urlcolor  = blue, 
linkcolor    = blue, 
citecolor   = blue 
}

\def\beq{\begin{equation}}
\def\eeq{\end{equation}}

\newtheorem{Theorem}{Theorem}[section]
\newtheorem{Remark}{Remark}[section]
\newtheorem{Lemma}{Lemma}[section]
\newtheorem{Definition}{Definition}[section]
\newtheorem{Proposition}{Proposition}[section]

\numberwithin{equation}{section}

\allowdisplaybreaks[4]
\begin{document}
\setlength{\columnsep}{5pt}
\title{On linearization of biholomorphism with non-semi-simple linear part at a fixed point}

\author{Yue Mi$^\dag$}
\address{Universit\'e C\^ote d'Azur, CNRS, Laboratoire J. A. Dieudonn\'{e}, 06108 Nice, France}
\email{yue.mi@univ-cotedazur.fr}
\author{Laurent Stolovitch$^{\dag\dag}$}
\address{Universit\'e C\^ote d'Azur, CNRS, Laboratoire J. A. Dieudonn\'{e}, 06108 Nice, France}
\email{Laurent.stolovitch@univ-cotedazur.fr}
\thanks{$\dag$ Y.Mi was supported by China Scholarship Council(CSC) (No.202006190129).}
\thanks{${\dag\dag}$The research of L. Stolovitch has been supported by the French government through the UCAJEDI Investments in the Future project managed by the National Research Agency (ANR) with the reference number ANR-15-IDEX-01.}
\begin{abstract}
We prove the holomorphic linearizability of germs of biholomorphisms of $(\mathbb{C}^n,0)$, fixing the origin, point at which the linear part has nontrivial Jordan blocks under the following assumptions~: We first assume the eigenvalues are of modulus less or equal than $1$, and that they are non-resonant. We also assume that they satisfied not only a classical Diophantine condition but also new Diophantine-like conditions related to quasi-resonance phenomena.
\end{abstract}

\date{}
\maketitle
\section{Introduction and main results}
Let $F$ be a germ of biholomorphism of $(\mathbb{C}^n,0)$ fixing the origin and let $A=F'(0)$ be its linear part (Jacobian matrix) at the origin. Let $\operatorname{spec}(A)=\{\lambda_1,\lambda_2,\cdots,\lambda_n\}$ be eigenvalues of $A$. We define the set of {\it resonant multi-indices} of $\operatorname{spec}(A)$ as follows~:
\begin{equation*}
\begin{aligned}
\operatorname{Res}[\operatorname{spec}(A)] &= \bigcup_{j=1}^{n} \operatorname{Res}^j[\operatorname{spec}(A)],  \\
\operatorname{Res}^j[\operatorname{spec}(A)] &= \{\alpha \in \mathbb{Z}_+^n(2) :\lambda^{\alpha}-\lambda_j =0 \},
\end{aligned}
\end{equation*}
where $\mathbb{Z}_+^n(k)=\{\alpha \in \mathbb{Z}_+^n: |\alpha|\geq k \}$, $\lambda=(\lambda_1,\lambda_2,\cdots,\lambda_n)$ and $\lambda^{\alpha}=\lambda_1^{\alpha_1}\cdots \lambda_n^{\alpha_n}$ for $\alpha=(\alpha_1,\alpha_2,\cdots,\alpha_n)\in \mathbb{Z}_+^n$ and $|\alpha|=\sum\limits_{i=1}^{n}|\alpha_i|$.

We say that $F$ is formally linearizable at the origin if there exists a formal power series transformation, fixing the origin, which is tangent to the identity
$\Phi(z)= z+ \varphi_{\geq 2}(z)\in (\mathbb{C}[[z]])^n$ such that
\begin{equation}\label{linearization}
\Phi^{-1} \circ F \circ \Phi(z) = Az.
\end{equation}
Here, $\varphi_{\geq 2}$ vanishes at order $\geq 2$ at the origin. If $A$ is diagonal, the well known Poincar\'{e}-Dulac theorem \cite{Arn83} ensures that if there is no resonance (i.e., $\operatorname{Res}[\operatorname{spec}(A)]=\emptyset$), then $F$ is formally linearizable. As the coefficients of the formal solution of (\ref{linearization}) involves products of form $(\lambda^{\alpha}-\lambda_j)^{-1}$, when the degree $|\alpha|$ goes to infinity, its convergence in a neighborhood of the origin is very related to the so-called {\it small divisors problem} occurring when $\inf_{\alpha}|\lambda^{\alpha}-\lambda_j|=0$ for some $j\in \{1,\cdots,n\}$.

The pioneering work of C.L.Siegel \cite{Sie42}, followed by E.Zehnder \cite{Zeh77}, shows that the Diophantine condition (for some fixed positive $C_0,\sigma$)
\begin{equation}\label{smalldivisor}
|\lambda^{\alpha}-\lambda_j|> C_0|\alpha|^{-\sigma}\quad \text{for all}\quad j=1,2,\cdots,n,\quad \alpha\in \mathbb{Z}_+^n(2)
\end{equation}
is sufficient to ensure the analyticity of $\Phi$ at the origin. This Diophantine condition has been weakened by  H.R\"ussmann \cite{Ru02} (and by A.D.Brjuno \cite{Bru71} for vector fields) to~: for all $m\geq 2$,
$$
|\lambda^{\alpha}-\lambda_j|\geq \frac{1}{\Omega(m)} \quad \text{for all}~1\leq j\leq n ~\text{and}~ |\alpha|=m,
$$
where $\Omega:\mathbb{N}\rightarrow \mathbb{R}$ is a function satisfying~:for all $m\in\mathbb{N}$,
\begin{equation*}
		m\leq \Omega(m)\leq \Omega(m+1), \quad\sum_{m=1}^{\infty}\frac{\operatorname{Log}\Omega(m)}{m^2} <\infty.
\end{equation*}
It is a major achievement due to J.-C. Yoccoz who proved the necessity of this condition for holomorphic linearization of non-resonant biholomorphism of $(\mathbb{C},0)$\cite{Yoc95Asterique}.
While previous results concern the linearization problem of a single holomorphic map, T.Gramchev and M.Yoshino \cite{GY99}, L. Stolovitch \cite{Sto15}  also obtained results on simultaneous linearization of a family of commuting biholomorphisms of $(\mathbb{C}^n,0)$.

All of results above require the linear part of the holomorphic map to be semi$-$simple. Little is known about the (non)linearizability of $F$ in the analytic category. We mention the work of T.Ueda \cite{Ue99} in which a new proof of the holomorphic conjugacy to a ``lower triangular polynomial map" with contracting linear parts is given, that is when
\begin{equation}\label{UEDA}
	\max\limits_{1\leq j\leq n}|\lambda_j|<1.
\end{equation}
In particular, this ensures the holomorphic linearization in the non-resonnant case. In dimension $n=2$, J.-C.Yoccoz \cite{Yoc95Asterique}[pp.86-87] proved that in general the analytic linearization can not be achieved when the linear part is a single Jordan block associated to an eigenvalue on the unit circle. In dimension $n=3$ and $n=4$, D.Delatte and T.Gramchev \cite{DG02} gave a positive answer for biholomorphic germs whose linear parts have one nontrivial $2$-dimensional Jordan block, provided that their eigenvalues satisfy some non-resonant and Diophantine-like conditions. They also gave an example to show the need of an arithmetic condition for the convergence to hold. Their proofs rely on very explicit computations of the solution of the {\it homological equation}.

The main purpose of the present paper is to extend their positive answers to any dimension with linear parts having possibly nontrivial Jordan blocks of any dimension. In this situation, the explicit computations of Delatte-Gramchev cannot be carried on and a more conceptual framework had to be developed for that purpose. In what follows, we assume
\begin{equation}
\max_{1\leq i\leq n}|\lambda_i|\leq 1 .
\end{equation}

Let us state our main theorem. If the linear part $F'(0)$ of $F$ is not semi-simple, then a preliminary linear change of variables allows us to assume that it is a typical Jordan normal form~:
\begin{equation}\label{Jn}
F'(0)=\Lambda^{\epsilon}= \Lambda+\epsilon N ,
\end{equation}
where $\Lambda= \operatorname{diag}\{\lambda_1,\cdots,\lambda_n\}$ is the diagonal matrix, $N$ is an upper triangular nilpotent matrix where all non-zero entries are $1$ and lie on the upper-diagonal. By dilation of coordinates, $\epsilon\neq 0$ can be made arbitrarily small. Let $\{\mu_i\}_{i=1}^m$ be the set of distinct modulus of its eigenvalues, and without loss of generality we can suppose that~:
\begin{equation}\label{modulus}
1=\mu_1>\mu_2>\cdots>\mu_m.
\end{equation}
We emphasize that $1$ is a modulus since $\Lambda^{\epsilon}$ is a not a strict contraction. Denote $\{1,2,\cdots,n\}=\bigcup_{i=1}^m\mathcal{I}_i$, for each $1\leq i\leq m$, $j\in \mathcal{I}_i$ when $|\lambda_j|=\mu_i$. Thus $\mathcal{I}_i\bigcap \mathcal{I}_j=\emptyset$ if $i\neq j$. For a fixed multi-index $\alpha\in \mathbb{N}^n$, we write $(\alpha_1,\alpha_2,\cdots,\alpha_n)=(A_1,A_2,\cdots,A_m)$ and $A_i=(\alpha_{\scriptscriptstyle l_i+1},\alpha_{\scriptscriptstyle l_i+2},\cdots,\alpha_{\scriptscriptstyle l_i+{n_i}})$, where
$$
  l_1 := 0, l_{i+1} := l_i+n_i, n_i:=|\mathcal{I}_i|,~\text{for}~i\leq m.
$$
Thus we have $|A_i|=\sum_{j=1}^{n_i}\alpha_{\scriptscriptstyle l_i+j}$ and $n=\sum_{i=1}^{m}n_i$. Correspondingly, we can also denote $(\lambda_1,\cdots,\lambda_n)=(\Lambda_1,\cdots,\Lambda_m)$ and $(z_1,\cdots,z_n)=(Z_1,\cdots,Z_m)$ in a similar way. Thus for all $i=1,\cdots,m$ we have:
\begin{equation}\label{l_i+j}
\Lambda_i^{A_i}=\prod_{j=1}^{n_i}\lambda_{\scriptscriptstyle l_i+j}^{\alpha_{\scriptscriptstyle l_i+j}}~~,~~Z_i^{A_i}=\prod_{j=1}^{n_i}z_{\scriptscriptstyle l_i+j}^{\alpha_{\scriptscriptstyle l_i+j}}.
\end{equation}
If there is no confusion, for a fixed $1\leq i\leq m$, we will usually write $i_j$ for $l_i+j$ so that $\alpha_{\scriptscriptstyle i_j}$ stands for $\alpha_{\scriptscriptstyle l_i+j}$ for $j=1,2,\cdots,n_i$.
\begin{Remark}
Notice that if there is a tuple of non-negative integers $\{\kappa_i\}_{1\leq i\leq m}$ such that $\sum_{i=1}^m\kappa_i\geq 2$ and $\prod_{i=1}^{m}\mu_i^{\kappa_i}=\mu_d$ for some $3\leq d\leq m$, then we must have $\kappa_j=0$ for all $j\geq d$.
\end{Remark}
With these notations, we can introduce the definition of the {\it quasi$-$resonance}~:
\begin{Definition}\label{defqs}
Let $\{\mu_i\}_{i=1}^m$ be distinct positive numbers satisfying $(\ref{modulus})$. For each $3\leq i\leq m$, we call Quasi-Resonance (w.r.t $\mu_i$) a relation of the form if~:
\begin{equation}\label{eq_qs}
\prod_{j=2}^{i-1}\mu_j^{\kappa_j}=\mu_i,
\end{equation}
for some tuple of non-negative integers $\{\kappa_j\}_{2\leq j\leq i-1}$. We also assume that
for all $k\in \mathcal{I}_i$, for all $A_j\in \mathbb{N}^{n_j}$ such that $|A_j|=\kappa_j$ and $\sum_{j=1}^{i-1}|A_j|\geq 2$, $2\leq j\leq i-1$, we have
\begin{equation*}
\Lambda_1^{A_1}\cdot\prod_{j=2}^{i-1}\Lambda_j^{A_j}\neq \lambda_k.
\end{equation*}

\end{Definition}
In order to achieve the convergence of the linearization, we need to give the Diophantine-like assumptions for each quasi-resonance.
\begin{Definition}
We say that $\Lambda^{\epsilon}$ satisfies a QR-Diophantine condition ($QR$ stands for Quasi-Resonant) if the following holds~:
For each $3\leq i\leq m$ where there is a quasi-resonance $($w.r.t $\mu_i)$, we have~:
\begin{equation}\label{def_qr_Dio}
\left|\Lambda_1^{A_1}\cdot\prod_{j=2}^{i-1}\Lambda_j^{A_j}- \lambda_k\right|\geq C_0\left(|A_1|+\sum_{j=2}^{i-1}\kappa_j\right)^{-\sigma},
\end{equation}
\end{Definition}
for all $k\in\mathcal{I}_i,\sum_{j=1}^{i-1}|A_j|\geq 2$ and $|A_j|=\kappa_j,2\leq j\leq i-1$.

\begin{Theorem}\label{mainth1}
Let $F(z)=\Lambda^{\epsilon} z+f(z)$ be a holomorphic map in a neighborhood of 0 in $\mathbb{C}^n$, $f$ has the order greater or equal than $2$ at the origin. Assume the linear part $(\ref{Jn})$ is non-resonant and $\epsilon\neq 0$ is sufficiently small, then there is a unique transformation $z=\Phi(\zeta)=\zeta +\phi_{\geq 2}(\zeta)$, biholomorphic in a neighborhood of 0, which solves the linearization problem $F\circ \Phi=\Phi\circ \Lambda^{\epsilon}$ near the origin, if one of the following conditions is satisfied~:
\begin{enumerate}[(I)]
  \item If there is no quasi-resonance, then there exists two positive numbers $C_0,\sigma>0$ such that for all $\lambda_i,i\in \mathcal{I}_1$,
  \begin{equation}\label{small_A1}
  \left|\Lambda_1^{A_1}-\lambda_i\right|\geq C_0(|A_1|)^{-\sigma},\quad\forall~|A_1|\geq 2,
  \end{equation}
  and for any $i,j\in \mathcal{I}_k,k\geq 2$, we have
  \begin{equation}\label{small_A2}
  \left|\Lambda_1^{A_1}\lambda_i-\lambda_j\right|\geq C_0(|A_1|+1)^{-\sigma},\quad\forall~|A_1|\geq 1.
  \end{equation}
  \item If there is a quasi-resonance, then besides $(\ref{small_A1})$ and $(\ref{small_A2})$, we also assume $\Lambda^{\epsilon}$ satisfies the $\mathrm{QR}-$Diophantine condition $(\ref{def_qr_Dio})$.
\end{enumerate}
\end{Theorem}

\begin{Remark}
Assuming we have $(\ref{modulus})$, then there are only finite numbers of quasi-resonance. In fact, for each $i=3,4,\cdots,m$, there are at most finite tuples of non-negative integers $\{\kappa^{(i)}:\kappa^{(i)}\in\mathbb{N}^{i-2}\}$ satisfying $(\ref{eq_qs})$.
\end{Remark}
The main idea for proof of the convergence is to estimate the solutions of {\it homological equations} $\mathcal{L}(f)=g$, that is the ``linearized version" of the conjugacy equation. In order to do that, we decompose the set of monomials along two sets called {\it Poincar\'e slice} and {\it Siegel slice} respectively. We then decompose the Taylor expansion of any germ of holomorphic function at the origin accordingly. We show that these slices are invariant sets of the homological operator. This enables to solve the projected equations $\mathcal{L}(f_P)=g_{\scriptscriptstyle P}$ and $\mathcal{L}(f_S)=g_{\scriptscriptstyle P}$. The aim of this decomposition is that on the Siegel slice we will encounter small divisors but we have arithmetic conditions to control it, while on the Poincar\'{e} slice, we have the complicated remainders brought by Jordan blocks but we do not encounter small divisors so that we have a uniform bound for the solution to the homological equation restricted to that set. The details will be seen in section \ref{re_decom}.

The paper is organized as follows: In section \ref{Notation} we introduce some basic notations and definitions of norms that we would need to use in the KAM scheme. Section \ref{re_decom} introduces the homological operator as well as {\it Poincar\'e slice} and {\it Siegel slice}. In section \ref{estremain}, we deal with the estimation of the remainders in the homological equation which is brought by nontrivial Jordan blocks. In section \ref{estofL}, we give the most important estimation of the solution to the homological equation, which will be used several times in the proof of the convergence in the section \ref{convergenceproof}.

\section{Notations and preliminaries}\label{Notation}
For an integer $k\geq 1$, and $\rho> 0,0<\delta<r\leq1$,
\begin{itemize}
\item $\mathcal{P}_k^n$ denotes the $\mathbb{C}-$ space of homogeneous polynomial vector fields on $\mathbb{C}^n$ and of degree $k$,
\item $p_k^n$ denotes the $\mathbb{C}-$ space of homogeneous polynomials on $\mathbb{C}^n$ of degree $k$,
\item $\mathcal{O}^n$ denotes the ring of germs at $0$ of holomorphic functions in $\mathbb{C}^n$,
\item $\mathbb{C}[[z]]_{\geq k}=\left\{f\left|\right.f(z)=\sum\nolimits_{\alpha\in \mathbb{Z}_+^n(k)}f_{\alpha}z^{\alpha},\quad f_{\alpha}\in \mathbb{C}\right\}$.
\end{itemize}

When $z\in\mathbb{C}^n,|z|:=\max_{1\leq i\leq n}|z_i|$. Denote $B_R=\{z\in\mathbb{C}^n:|z_i|\leq r_i,\forall~ 1\leq i\leq n\}$ for $R=(r_1,r_2,\cdots,r_n)$ with positive numbers $r_i$. For a formal power series $g(z)\in\mathbb{C}[[z]]$ denote $\bar{g}(z)=\sum\nolimits_{\alpha\in \mathbb{Z}_+^n}|g_{\alpha}|z^{\alpha}$, the common polydisc norm is $\bm{|}g\bm{|}_R=\bar{g}(r_1,r_2,\cdots,r_n)$ corresponding with $B_R$.
We set an asymmetric ball: for two disjoint sets $D$ and $J$ such that $D\bigcup J=\{1,2,\cdots,n\}$, where actually $\lambda_j,j\in J$ corresponds to the Jordan block. Denote
\begin{equation}\label{B_rrho}
B_{r,\rho}=\{z\in \mathbb{C}^n: |z_d|< r, d\in D ~~\text{and}~~ |z_j|<\rho, j\in J\}, \quad B_r:=B_{r,r}.
\end{equation}
We introduce two Banach algebras~:
\begin{equation*}
H(B_{r,\rho})=\left\{g(z)=\sum\nolimits_{\alpha\in \mathbb{Z}_+^n}g_{\alpha}z^{\alpha}, g_{\alpha}\in \mathbb{C}: \bm{|}g\bm{|}_{r,\rho}<+\infty \right\},
\end{equation*}
with the polydisc norm corresponding to $B_{r,\rho}$ :
\begin{equation*}
\bm{|}g\bm{|}_{r,\rho} := \bar{g}(r,\rho), ~~z_i=r ~~\text{if}~i\in I,~ z_j=\rho ~~\text{if}~j\in J.
\end{equation*}
Let $D\supseteq B_{r,\rho}$  be an open set.  Let $\mathcal{O}_{2}(D)$ be the space of all holomorphic functions in $D$ such that $f(0)=0, \partial_{z_j}f(0)=0, j=1,2,\cdots,n$. Let us define
\begin{equation*}
\mathcal{H}(D)=\left\{f \in C(\overline{D}) \cap\mathcal{O}^n_{2}(D):\|f\|_{D}<+\infty\right\},\quad \|f\|_D=\sup _{z \in D}|f(z)|.
\end{equation*}
For all vector-valued functions $F=(f_1,\ldots, f_n$), the norm $|F|_{r,\rho}$ (resp. $\|F\|$) denote the maximum of the norm of its components. To simplify the notation, we also denote $\|F\|_r = \|F\|_{B_r}$.
By Cauchy estimates, we have~:
\begin{equation}\label{polyineq}
\|g\|_{r-\delta}\leq \bm{|}g\bm{|}_{r-\delta}\leq \left(\frac{r}{\delta}\right)^n \|g\|_r.
\end{equation}
For an $n\times n$ matrix $A$, we set $|A|:=\max\limits_{1\leq i,j\leq n}|a_{ij}|$.  For $z\in\mathbb{C}^n$, we have~:
\begin{equation}\label{matrix_mod}
|Az|\leq n|A||z|.
\end{equation}
We introduce two basic inequalities (e.g. \cite{Sto00}[p.149]) which will be used in section \ref{estremain}:
\begin{equation}\label{cauchyest}
\begin{aligned}
\bm{|}f\bm{|}_{r-\delta}&\leq \left(\frac{r-\delta}{r}\right)^k \bm{|}f\bm{|}_r,  ~\text{if}~ \text{ord(f)}\geq k ,\\
\left|\frac{\partial f}{\partial z_i}\right|_r &\leq \frac{k}{r} \bm{|}f\bm{|}_r, ~\text{if}~ f\in \mathcal{P}_n^k.
\end{aligned}
\end{equation}
We shall write $f_{\geq k}$ to emphasize that $f$ has order $\geq k$ at the origin.

At last, we introduce a domain and a norm which is need in the proof of convergence (here $\epsilon$ is fixed)~:
$$D_r:=B_r\cup \Lambda^{\epsilon} (B_r),\quad \|g\|_{C^1,r}:= \max\{\|g\|_{D_r}, \|Dg\|_{D_r} \}.$$

\section{Representation of the homological operator and its decomposition}\label{re_decom}
Let $F(z)=\Lambda^{\epsilon}z+f_{\geq 2}(z)$ be a germ of biholomorphism of $(\mathbb{C}^n,0)$. Let $\Phi(y)=y+\phi(y)$ be a linearizing change of variables~: $F\circ \Phi=\Phi(\Lambda^{\epsilon}z)$. Hence, $\phi$, which vanishes at order $\geq 2$ at the origin,  solves  $\phi\circ \Lambda^{\epsilon}-\Lambda^{\epsilon}\phi=f_{\geq 2}(I+\phi)$. Hence, it is natural to introduce the {\it homological equation/operator}~:
\begin{equation}\label{Lphi_g}
\mathcal{L}(\varphi):=\varphi\circ \Lambda^{\epsilon}-\Lambda^{\epsilon}\varphi= g ,\quad\quad\quad g\in H(B_r)^n.
\end{equation}
We want first to find a  formal solution $\varphi\in\mathbb{C}[[z]]^n$ and then we want to estimate $\varphi$ with respect to a given $g$. 
With the above settings, we can rewrite $\mathcal{L}(\varphi)$ in (\ref{Lphi_g}) as :
\begin{equation}\label{decomLphi_g}
\mathcal{L}(\varphi)= (\Omega-\epsilon N)\varphi+ \mathcal{R}^{\epsilon}\varphi :=\mathcal{L}^{\epsilon}\varphi+\mathcal{R}^{\epsilon}\varphi ,
\end{equation}
where
\begin{equation}\label{Ldecomp}
\Omega(\varphi)= \varphi \circ \Lambda - \Lambda \varphi,\quad \mathcal{R}^{\epsilon}(\varphi) = \varphi \circ (\Lambda+\epsilon N)-\varphi \circ \Lambda.
\end{equation}
Given a formal power series $f(z)=\sum_{\alpha}f_{\alpha}z^{\alpha}$,  we have
\begin{equation}\label{OmegaN}
\Omega (f)(z)= \sum_{\alpha} (\Omega_{\alpha}f_{\alpha})z^{\alpha},\quad N(f)(z)= \sum_{\alpha} (N f_{\alpha})z^{\alpha},
\end{equation}
where $\Omega_{\alpha}= \operatorname{diag}\{\lambda^{\alpha}-\lambda_1,\cdots,\lambda^{\alpha}-\lambda_n\}$ is a diagonal matrix. As $\Lambda^{\epsilon}$ is Jordan matrix, the matrix $\Omega_{\alpha}$ and $N$ are commuting for each $\alpha\in \mathbb{Z}_+^n(2)$. Thus due to the nilpotency of $N$ and the non-resonance of the eigenvalues of $\Lambda^{\epsilon}$, $(\Omega-\epsilon N)$ is invertible at the formal level (i.e. on $\mathbb{C}[[z]]_{\geq 2}^n$), we have
\begin{equation}\label{inverse-d}
(\Omega_{\alpha}-\epsilon N)^{-1} = \Omega^{-1}_{\alpha}+\sum_{s=1}^{d-1}\epsilon^s \Omega^{-s-1}_{\alpha}N^s ,
\end{equation}
where $d$ is the nilpotency order of $N$.

Hence, emphasizing the {\it remainder operator} $\mathcal{R}^{\epsilon}$, equation (\ref{Lphi_g}) reads:
\begin{equation}\label{inverse-phi}
\varphi = (\Omega-\epsilon N)^{-1}g - (\Omega-\epsilon N)^{-1} \mathcal{R}^{\epsilon}\varphi.
\end{equation}

The abstract equation above plays an important role in the estimation of the solution to the homological equation. The main difficulty in the Jordan block case is to estimate the remainder term $\mathcal{R}^{\epsilon}$. To overcome this difficulty, we introduce the most important idea : (S,P)-decomposition, which is inspired by D.Delatte and T.Gramchev \cite{DG02}[p.10, Definition 2.4]. The aim of this decomposition is that on the Siegel slice we will encounter small divisors but we have arithmetic conditions to control it, while on the Poincar\'{e} slice, although we have complicated remainders brought by Jordan blocks, we do not encounter small divisors and we will have a uniform bound for the solution to the homological equation restricted on it. Briefly speaking, we want to find two disjoint subsets $S$ and $P$ of $\mathbb{Z}^n_+(2)$ such that the following conditions holds:
\begin{enumerate}[i)]
  \item $\mathbb{Z}^n_+(2)= S\bigcup P,\quad S\bigcap P=\varnothing , $
  \item Invariance: $\mathcal{L}^{\pm 1} S \subset S ,\quad \mathcal{L}^{\pm 1} P \subset P$.
  \item The inverse $\mathcal{L}^{-1}_{\scriptscriptstyle S}$ of restriction of the operator on Siegel slice $\mathcal{L}_{\scriptscriptstyle S}:=\mathcal{L}|_{\scriptscriptstyle S}$, is related to small divisors problem, but Diophantine conditions for $\alpha\in S$ gives a control;
  \item Poincar\'e slice $\mathcal{L}_{\scriptscriptstyle P}:=\mathcal{L}|_{\scriptscriptstyle P}$ does not involve in any small divisors, thus we have a uniform estimate of its inverse $\mathcal{L}^{-1}_{\scriptscriptstyle P}$.
\end{enumerate}
We rewrite $\varphi \in \mathbb{C}[[z]]^n$ along such a decomposition~:
\begin{equation}\label{P_S}
\begin{gathered}
\varphi(z)=\sum_{\alpha\in \mathbb{Z}_+^n(2)}\varphi_{\alpha}z^{\alpha}= \varphi_{\scriptscriptstyle S}(z)+\varphi_{\scriptscriptstyle P}(z),\\
\varphi_{\scriptscriptstyle S}(z)=\sum_{\alpha\in S}\varphi_{\alpha}z^{\alpha}, \quad
\varphi_{\scriptscriptstyle P}(z)=\sum_{\alpha\in P}\varphi_{\alpha}z^{\alpha}.
\end{gathered}
\end{equation}
By abuse of notations, we shall write that $\varphi(z)\in S$ if for all all monomials (of the Taylor expansion at $0$) of $\varphi(z)$ belong to $S$.

Let us define the sets $S$ and $P$ according to the diffrent cases of Theorem \ref{mainth1}~:

In case $(\uppercase\expandafter{\romannumeral1})$, as there is no quasi-resonance, we decompose $\mathbb{Z}_+^n(2)$ for $\Lambda^{\epsilon}$ in (\ref{Jn}) as follows~:
\begin{equation}\label{decomp_nqr}
\begin{gathered}
S:=\{\alpha\in \mathbb{Z}_+^n(2): \sum_{i=2}^{m}|A_i|\leq 1 \},\\
P:=\{\alpha\in \mathbb{Z}_+^n(2): \sum_{i=2}^{m}|A_i|\geq 1  \}.
\end{gathered}
\end{equation}
In case $(\uppercase\expandafter{\romannumeral2})$, there are quasi-resonances. 
 Let us denote
\begin{equation*}
QR_i=\left\{\kappa^{(i)}=(\kappa^{(i)}_2,\kappa^{(i)}_3,\cdots,\kappa^{(i)}_{i-1})\in \mathbb{N}^{i-2}: \prod_{j=2}^{i-1}\mu_j^{\kappa^{(i)}_j}=\mu_i\right\}
\end{equation*}
the set of quasi-resonances w.r.t $\mu_i$. According to $(\ref{modulus})$, then for each $3\leq i\leq m$, there exists at most finite quasi-resonances, i.e., $\sharp|QR_i|< \infty$. By non-resonance condition we also have:
\begin{equation*}
\Lambda_1^{A_1}\cdot\prod_{j=2}^{i-1}\Lambda_j^{A_j}\neq \lambda_k,
\end{equation*}
for all $k\in\mathcal{I}_i, \sum_{j=1}^{i-1}|A_j|\geq 2$ and $|A_j|=\kappa^{(i)}_j,2\leq j\leq i-1$.
Denote the subsets of $\mathbb{Z}_+^n(2)$ as the following:
\begin{equation}\label{s_1s_2}
\begin{aligned}
S_1&:=\{\alpha\in \mathbb{Z}_+^n(2): |A_j|=0~,\text{for}~2\leq j\leq m\},\\
S_2&:=\{\alpha\in \mathbb{Z}_+^n(2): \sum_{l=2}^{m}|A_j|=1\},
\end{aligned}
\end{equation}
and for any $\kappa\in QR_i$,
\begin{equation}\label{s_ik}
\begin{aligned}
S_{i,\kappa}&:=\{\alpha\in \mathbb{Z}_+^n(2):|A_j|=\kappa_j~\text{for}~2\leq j\leq i-1,
|A_j|=0~\text{for}~i\leq j\leq m\},\\
S_i&:=\bigcup_{\kappa\in QR_i}S_{i,\kappa},\quad 3\leq i\leq m.
\end{aligned}
\end{equation}
We define the $(S,P)-$decomposition to be~:
\begin{equation}\label{decomposition}
S:=\bigcup_{i=1}^{m}S_i,\quad P:=\bigcap_{i=1}^{m}S_i^c,
\end{equation}
where $S_i^c$ denotes the complement of $S_i$ in $\mathbb{Z}_+^n(2)$.
\begin{Remark}
The idea for defining these slices is to select $\alpha\in \mathbb{Z}_+^n(2)$ such that $|\lambda^{\alpha}|=|\lambda_i|$.
\end{Remark}

Let us check that the first two properties $\romannumeral1),\romannumeral2)$ of our decomposition. The last two properties will be seen in the following sections.
\begin{Lemma}\label{invariance}
The homological operator $(\ref{decomLphi_g})$ as well as decomposition $(\ref{Ldecomp})$ have the following invariance property~: For each Siegel slice $S_i$, $i=1,2,\cdots,m$, we have~:
\begin{equation*}
\mathcal{L}^{\pm1} S_i\subset S_i.
\end{equation*}
On Poincar\'e slice we have~:
\begin{equation*}
\mathcal{L}^{\pm1} P\subset P.
\end{equation*}
\end{Lemma}
\begin{proof}
Recalling (\ref{OmegaN}), we observe that $\Omega(\varphi)$ and $N(\varphi)$ preserve each monomial of $\varphi$. Hence the operators $\Omega$ and $N$ satisfy the invariance property. For the remainder part, first we need to know the specific form of the linear part $(\ref{Jn})$. Assume there is only one Jordan block corresponding to $\Lambda_t$, here we use the notations of $(\ref{l_i+j})$ and replace $l_t+j$ by $t_j$ for simplicity. Let us write
$$\varphi(z)=\sum\limits_{\alpha\in\mathbb{Z}_+^n(2)}\varphi_{\alpha}z^{\alpha}=\sum\limits_{A\in\mathbb{Z}_+^n(2)}\varphi_A Z^A.$$
By (\ref{Ldecomp}), we have~:
\begin{equation}\label{R_phi}
\begin{aligned}
&\mathcal{R}^{\epsilon}\varphi(z)
=\varphi \circ (\Lambda+\epsilon N)-\varphi \circ \Lambda \\
=&\sum_{A}\varphi_A\mathop{\prod}_{1\leq j\leq m \atop j\neq t}(\Lambda_jZ_j)^{A_j}
\left[\prod_{l=1}^{n_t-1}(\lambda_{t_l} z_{t_l}+\epsilon z_{t_{l+1}})^{\alpha_{t_l}}\cdot(\lambda_{t_{n_t}} z_{t_{n_t}})^{\alpha_{t_{n_t}}}-
\prod_{l=1}^{n_t}(\lambda_{t_l} z_{t_l})^{\alpha_{t_l}}\right].
\end{aligned}
\end{equation}
We claim that  the operator $\mathcal{R}^{\epsilon}$ does not change the total degree of those variables corresponding to the Jordan block and it keeps each monomial of those variables corresponding to the diagonal eigenvalues.\\
In fact, for a monomial $z_{t_1}^{\alpha_{t_1}}\cdots z_{t_{n_t}}^{\alpha_{t_{n_t}}}$ the variables of which correspond to $\Lambda_t$, assume its total degree $|A_t|=\sum_{l=1}^{n_t}\alpha_{t_l}=\kappa$ for some $\kappa$. By binomial expansion, for each $1\leq l\leq n_t-1$ we have~:
\begin{equation}\label{binomial_l}
(\lambda_{t_l} z_{t_l}+\epsilon z_{t_{l+1}})^{\alpha_{t_l}}=\sum_{s_l=0}^{\alpha_{t_l}}\binom{\alpha_{t_l}}{s_l}(\lambda_{t_l} z_{t_l})^{s_l}(\epsilon z_{t_{l+1}})^{\alpha_{t_l}-s_l}.
\end{equation}
Let $\tilde{A_t}=(\tilde{\alpha}_{t_1},\cdots,\tilde{\alpha}_{t_{n_t}})$  be the exponent of one of the new monomials in the variables $z_{t_1},\cdots,z_{t_{n_t}}$ appearing in the bracket of the formula $(\ref{R_phi})$. Each new monomial coming from the product in the bracket in $(\ref{R_phi})$, is obtained by choosing different $s_l$ in $(\ref{binomial_l})$, that is,
\begin{equation*}
\begin{aligned}
\tilde{\alpha}_{t_1}&=s_1,\\
\tilde{\alpha}_{t_{l+1}}&=\alpha_{t_l}-s_l+s_{l+1},\\
\tilde{\alpha}_{t_{n_t}}&=\alpha_{t_{n_t-1}}-s_{n_t-1}+\alpha_{t_{n_t}},
\end{aligned}
\end{equation*}
for $1\leq l\leq n_t-2$. This directly shows~:
\begin{equation*}
\begin{aligned}
|\tilde{A}_t|=\sum_{l=1}^{n_t}\tilde{\alpha}_{t_l}=&s_1+\sum_{l=1}^{n_t-2}(\alpha_{t_l}-s_l+s_{l+1})+\alpha_{t_{n_t-1}}-s_{n_t-1}+\alpha_{t_{n_t}}\\
=&\sum_{l=1}^{n_t}\alpha_{t_l}=|A_t|=\kappa.
\end{aligned}
\end{equation*}
Thus the total degree of all of the new monomials corresponding to $\Lambda_t$ does not change and the degree of each one of the other variables is kept fixed, which proves our claim.

Similarly, if there are two blocks corresponding to $\Lambda_t$ and $\Lambda_s$, then
\begin{align*}
\mathcal{R}^{\epsilon}\varphi(z)
&=\varphi \circ (\Lambda+\epsilon N)-\varphi \circ \Lambda \\
=\sum_{\alpha}\varphi_{\alpha}&\mathop{\prod}_{1\leq j\leq m \atop j\neq t,s}(\Lambda_jZ_j)^{A_j}
\prod_{l=1}^{n_t-1}(\lambda_{t_l} z_{t_l}+\epsilon z_{t_{l+1}})^{\alpha_{t_l}}\cdot(\lambda_{t_{n_t}} z_{t_{n_t}})^{\alpha_{t_{n_t}}} \\
\cdot &\prod_{l=1}^{n_s-1}(\lambda_{s_l} z_{s_l}+\epsilon z_{s_{l+1}})^{\alpha_{s_l}}\cdot(\lambda_{s_{n_s}} z_{s_{n_s}})^{\alpha_{s_{n_s}}} \\
-\sum_{\alpha}\varphi_{\alpha}&\mathop{\prod}_{1\leq j\leq m \atop j\neq t,s}(\Lambda_jZ_j)^{A_j}
\prod_{l=1}^{n_t}(\lambda_{t_l} z_{t_l})^{\alpha_{t_l}}\cdot
\prod_{l=1}^{n_s}(\lambda_{s_l} z_{s_l})^{\alpha_{s_l}}.
\end{align*}
Observe that if we restricted the operator on
\begin{equation}\label{A_former}
|A_t|=\kappa_t,|A_s|=\kappa_s,
\end{equation}
for some fixed integers $\kappa_t,\kappa_s$, then for the new exponents of monomials produced, $|\tilde{A_t}|,|\tilde{A_s}|$ will be maintained, i.e.,
\begin{equation}\label{A_latter}
|\tilde{A_t}|=\kappa_t,|\tilde{A_s}|=\kappa_s.
\end{equation}
Actually we find that once the linear part has a Jordan block, the total degree of the corresponding monomials will always be maintained, moreover, each monomial of the variables corresponding to diagonal part is kept fixed. Considering the way of choosing the Siegel slice $(\ref{s_1s_2})$ and $(\ref{s_ik})$, we can deduce that for each $1\leq i\leq m$, $(\mathcal{R}^{\epsilon})^{\pm1}S_i\subset S_i$. Thus actually we have :
\begin{equation*}\label{LS_i_in_S_i}
\mathcal{L}^{\pm1} S_i\subset S_i.
\end{equation*}

On Poincar\'e slice, considering the choice of the Siegel slice $(\ref{s_1s_2})$ and $(\ref{s_ik})$ are all based on $|A_j|=a$ for some fixed $a=0,1$ or $a=\kappa_j$ in $(\ref{s_ik})$. Combining $(\ref{A_former})$ and $(\ref{A_latter})$, when an $\alpha\in S_1^c\bigcap S_2^c$, we have the corresponding indices $|\tilde{A}_j|\neq 0,1$ for one of $j\in \{2,\cdots,m\}$, thus
\begin{equation}\label{R_s1s2}
\mathcal{R}^{\epsilon}(S_1^c\bigcap S_2^c) \bigcap (S_1\bigcup S_2)=\emptyset.
\end{equation}
On the other hand, for each $ 3\leq i \leq m$  and  $\alpha\in S_{i,\kappa}^c$, we have either $|\tilde{A}_j|\neq \kappa_j$ for some $ j\in \{2,\cdots,i-1\}$, or $|\tilde{A}_j|\neq 0$ for some $j\geq i$. This means $\mathcal{R}^{\epsilon}(S_{i,\kappa}^c)\bigcap S_{i,\kappa}=\emptyset$, and naturally we have
\begin{equation*}
\mathcal{R}^{\epsilon}(\bigcap_{\kappa\in QR_i} S_{i,\kappa}^c)\subset \bigcap_{\kappa\in QR_i}\mathcal{R}^{\epsilon}( S_{i,\kappa}^c),
\end{equation*}
thus
\begin{equation*}\label{}
\mathcal{R}^{\epsilon}(\bigcap_{\kappa\in QR_i} S_{i,\kappa}^c) \bigcap (\bigcup_{\kappa\in QR_i}S_{i,\kappa})=\emptyset,\quad \text{i.e. }\mathcal{R}^{\epsilon}(S_i^c) \bigcap (S_i)=\emptyset.
\end{equation*}
Similarly we have
\begin{equation}\label{R_si}
\mathcal{R}^{\epsilon}(\bigcap_{i=3}^{m} S_i^c) \bigcap (\bigcup_{i=3}^{m} S_i)=\emptyset .
\end{equation}
Combining $(\ref{R_s1s2})$ and $(\ref{R_si})$ we have
\begin{equation}\label{R_s}
\mathcal{R}^{\epsilon}(\bigcap_{i=1}^{m} S_i^c) \bigcap (\bigcup_{i=1}^{m} S_i)=\emptyset .
\end{equation}
This directly shows that
\begin{equation*}
\mathcal{R}^{\epsilon}(P)=\mathcal{R}^{\epsilon}(\bigcap_{i=1}^{m}S_i^c)\subset \bigcap_{i=1}^{m}S_i^c=P.
\end{equation*}
Thus the lemma is proved.
\end{proof}
This shows that our $(S,P)-$decomposition satisfies the invariance property $\romannumeral1),\romannumeral2)$. The left two properties will be proved in the following sections. To summarize, $(S,P)-$decomposition allows us to consider it separately within the two slices in order to get a final convergence of the solution to the convergent linearization problem.

\section{Estimate of the remainder operator $\mathcal{R}^{\epsilon}$}\label{estremain}

In order to simplify the notations and statements, we will first consider the case of a single Jordan block. We assume it corresponds to $\Lambda_d$ for some fixed integer $d\geq 2$. In this case, the asymmetric ball $(\ref{B_rrho})$ reads~:
\begin{equation}\label{BallJn}
B_{r,\rho}=\{z\in \mathbb{C}^n :|z_i|<r ~\text{for}~i\notin \mathcal{I}_d~ , |z_j|<\rho ~\text{for}~j \in \mathcal{I}_d\}.
\end{equation}
Using notations from $(\ref{l_i+j})$, we will replace $l_i+j$ by $i_j$.

In this section we give a crucial estimate for the homological operator restricted to the Poincar\'{e} slice~:
\begin{Lemma}\label{lemmaR}
Let $\epsilon$ be small enough (depending only on $\Lambda$). Then there is a constant $\tilde{C}$ which depends only on $\Lambda$, such that, for any $\varphi \in H(B_{r,\rho})$,
\begin{equation}\label{est_remainder}
\bm{|}\mathcal{R}^{\epsilon}\varphi\bm{|}_{r,\rho}\leq \epsilon \tilde{C} \bm{|}\varphi\bm{|}_{r,\rho} .
\end{equation}
\end{Lemma}
\begin{proof}
First suppose there is only one single Jordan block in $(\ref{Jn})$. Then by Taylor formula we have
\begin{equation}\label{eqR}
  \mathcal{R}^{\epsilon}(\varphi)(z)= \varphi(\Lambda z+\epsilon Nz)-\varphi(\Lambda z )= \int_{0}^{1} D\varphi(\Lambda z+t\epsilon Nz)\cdot (\epsilon Nz) dt .
\end{equation}
Let us write $\varphi=(\varphi^{i})_{i=1,\cdots,n}$ with $\varphi^{i}=\sum_{k\geq 2}\varphi^{i}_k$ where $\varphi^{i}_k\in p_n^k$. Since $N$ is nilpotent, only those derivatives of the variables corresponding to $\Lambda_d$ make contributions to the integral in (\ref{eqR}). Since $|\lambda_{j}|< 1$, $j\in \mathcal{I}_d$, we may choose $\epsilon$ sufficiently small such that $|\lambda_j|+\epsilon < 1, j\in\mathcal{I}_d$. For all $0<t<1$ we have
\begin{equation*}
(\Lambda+t\epsilon N)B_{r,\rho}\subset \Lambda^{\epsilon}({B_{r,\rho}})\subset B_{r,(\mu_d+\epsilon) \rho}.
\end{equation*}
Applying the Cauchy estimates for polydisc norm (\ref{cauchyest}) to each partial derivative $\frac{\partial \varphi_k^i}{\partial z_j},j\in\mathcal{I}_d$ from $B_{r,(\mu_d+\epsilon)\rho}$ to $B_{r,\rho}$, for all $i=1,2,\cdots,n$ and $k\geq 2$~:
\begin{align}
\nonumber\bm{|}\mathcal{R}^{\epsilon}\varphi^i_k\bm{|}_{r,\rho}
&= \left|\int_{0}^{1} D\varphi^i_k(\Lambda z+t\epsilon Nz)\cdot (\epsilon Nz) dt \right|_{r,\rho} \\
\nonumber&\leq \left|\int_{0}^{1}\sum_{j=1}^{n_d-1}\frac{\partial \varphi^i_k}{\partial z_{d_j}}(\Lambda z+t\epsilon Nz) \cdot \epsilon z_{d_{j+1}}  dt\right|_{r,\rho}\\
&\leq \sum_{j=1}^{n_d-1}\left|\frac{\partial\varphi^i_k(\Lambda^{\epsilon})}{\partial z_{d_j}}\right|_{r,\rho} \cdot \left|\epsilon z_{d_{j+1}}\right|_{r,\rho}\leq \sum_{j=1}^{n_d-1}\left|\frac{\partial \varphi^i_k}{\partial z_{d_j}}\right|_{r,(\mu_d+\epsilon)\rho} \cdot \epsilon \rho \label{Rphi}\\
\nonumber&\leq \epsilon(n_d-1) \rho\frac{k}{(\mu_d+\epsilon)\rho}|\varphi^i_k|_{r,(\mu_d+\epsilon)\rho} \\
\nonumber&\leq \epsilon(n_d-1) \frac{k}{\mu_d+\epsilon} \left[\frac{(\mu_d+\epsilon)\rho}{\rho}\right]^k|\varphi^i_k|_{r,\rho}\\
\nonumber&= \epsilon(n_d-1)(\mu_d+\epsilon)^{k-1}k\cdot |\varphi^i_k|_{r,\rho}.
\end{align}
For a constant $a, 0<a<1$, we consider a function $h_a(x)=\ln(a^{x-1}x)=(x-1)\ln a+\ln x$, its derivative is $h_a'(x)=\ln a+ \frac{1}{x}$, so $h_a(x)$ has maximum at $x_0=-\frac{1}{\ln a}$. Let $k_0$ be the one who satisfies $h_a(k_0)=\max \{h_a([x_0]),h_a([x_0]+1),h_a(2)\}$, thus for all $k\geq 2$ we have $h_a(k)\leq h_a(k_0)$. Now if we suppose
\begin{equation}\label{e1}
\epsilon< \frac{1-\mu_d}{2},
\end{equation}
then for $\mu_d+\epsilon <\frac{1+\mu_d}{2}:=a<1$, we have for all $k\geq 2$~:
\begin{equation*}
(\mu_d+\epsilon)^{k-1}k \leq \exp{h_a(k_0)},
\end{equation*}
where $k_0$ and $a$ depend only on $\Lambda$. By this estimate and observing that $\varphi^{i}=\sum_{k\geq 2}\varphi^{i}_k$, for each $i=1,2,\cdots,n$ we obtain
\begin{equation}\label{tildeC}
\bm{|}\mathcal{R}^{\epsilon}\varphi^i\bm{|}_{r,\rho}\leq \epsilon \tilde{C}|\varphi^i|_{r,\rho},
\end{equation}
where the constant $\tilde{C}=\tilde{C}(n_d,\lambda_d)$ which depends only on $n_d$ and $\lambda_d$. With the estimates above, it can be seen easily that, in case of several Jordan blocks, we consider in $(\ref{B_rrho})$ with $|z_i|<r$ corresponding to variables of diagonal part and $|z_j|<\rho$ corresponding to variables in Jordan blocks. Then there will be several summations in equation $(\ref{Rphi})$, and $\tilde{C}$ will depend on the dimension and eigenvalues of those corresponding blocks. Thus the lemma is proved.
\end{proof}
\begin{Remark}
Although $\varphi$ does not need to be restricted on the Poincar\'e slice in the previous lemma, we shall only use this estimate on Poincar\'e slice.
\end{Remark}

\section{Estimate of the solution to the homological equation}\label{estofL}
In this section, for the convenience of convergence proof we will consider the ball $B_r=B_{r,r}$. Our main purpose is to solve and estimate the solution of the homological equation $(\ref{Lphi_g})$. According to invariance property $\romannumeral2)$ of $(S,P)-$decomposition, we consider it along the two separate slices~:
\begin{equation}\label{homode}
\begin{aligned}
\mathcal{L}(\varphi_{\scriptscriptstyle P}+\varphi_{\scriptscriptstyle S})= g_{\scriptscriptstyle P}+g_{\scriptscriptstyle S} \Rightarrow
\mathcal{L}(\varphi_{\scriptscriptstyle P})=g_{\scriptscriptstyle P},\mathcal{L}(\varphi_{\scriptscriptstyle S})=g_{\scriptscriptstyle S}.
\end{aligned}
\end{equation}
Suppose that the linear part $(\ref{Jn})$ satisfies the corresponding $QR-$Diophantine conditions, our main goal is to prove the following proposition :
\begin{Proposition}\label{ProJn}
Suppose $\epsilon$ is sufficiently small and $\Lambda^{\epsilon}$ satisfies the conditions in Theorem \ref{mainth1}, then there exists a constant $C_2$ which depends on $\kappa,\sigma,\Lambda,n,C_0$ so that for a holomorphic map $g_{\geq 2}\in {\mathcal{H}(B_r)}^n$, there is a unique solution $\varphi\in {\mathcal{H}(B_{r-\delta})}^n$ of the homological equation $(\ref{Lphi_g})$ which satisfies the following estimate~:
\begin{equation}\label{Est_L}
\begin{aligned}
\|\varphi\|_{r-\delta}\leq C_2\cdot \delta^{-(\vartheta+n)}\|g\|_r,\quad \text{for all}\quad 0<\delta< r\leq1,
\end{aligned}
\end{equation}
where $\vartheta$ is the positive constant in equation $(\ref{EstonS})$ which depends only on $\kappa,\sigma,n$.
\end{Proposition}
We need to state some basic lemmas first for its proof. First, on the Poincar\'e slice we have a formal solution and we estimate it.
\begin{Lemma}\label{lemmaP}
The restriction to the Poincar\'e slice of the homological equation $(\ref{homode})$ has a unique solution $\varphi_{\scriptscriptstyle P}\in H(B_r)^n$ w.r.t 
any given $g_{\scriptscriptstyle P}\in H(B_r)^n$. Suppose $\epsilon$ is sufficiently small, then we have the following estimate~:
\begin{equation}\label{PE}
\bm{|}\varphi_{\scriptscriptstyle P}\bm{|}_r \leq C_P \bm{|}g_{\scriptscriptstyle P}\bm{|}_r ,
\end{equation}
where $C_P=C(\Lambda,n,\kappa)$ depends only on $\Lambda,n$ and $\kappa$.
\end{Lemma}
\begin{proof}
First we suppose there is only one single Jordan block. The restriction of the homological equation (\ref{decomLphi_g}) to the Poincar\'e slice is~:
\begin{equation}\label{LonPoincare}
\mathcal{L}(\varphi_{\scriptscriptstyle P})
=(\Omega-\epsilon N)\varphi_{\scriptscriptstyle P}+ \mathcal{R}^{\epsilon}\varphi_{\scriptscriptstyle P}
= g_{\scriptscriptstyle P}.
\end{equation}
Throughout this proof, let $\alpha\in P\subset \mathbb{Z}_+^n(2)$. We claim that for each $i=1,2,\cdots,n$, there exists a constant $\gamma_i>0$ depending on $\Lambda$ and $\kappa$ such that
\begin{equation}\label{claim_P}
|\lambda^{\alpha}-\lambda_i|\geq \gamma_i.
\end{equation}
In fact, by the $(S,P)-$decomposition $(\ref{decomp_nqr})$ and $(\ref{decomposition})$, we have~:
\begin{equation}\label{lambda_P}
|\lambda^{\alpha}|\neq |\lambda_i|~\text{for all}~ i=1,2,\cdots,n.
\end{equation}
Equation $(\ref{modulus})$ indicates that $|\lambda^{\alpha}|$ tends to zero as $|\sum_{j\notin \mathcal{I}_1}\alpha_j|$ goes to infinity. Combining the non-resonant condition
\begin{equation}\label{lambda_P2}
|\lambda^{\alpha}-\lambda_i|\neq 0 ~\text{for all}~i=1,2,\cdots,n,
\end{equation}
and $(\ref{lambda_P}),$ there exists an integer $M_i$ such that we have $|\lambda^{\alpha}-\lambda_i|\geq \frac{1}{2}|\lambda_i|$ when $|\alpha|\geq M_i$. Denote
\begin{equation*}
\gamma_i :=\min\{\inf_{|\alpha|<M_i}|\lambda^{\alpha}-\lambda_i|,\frac{1}{2}|\lambda_i|\},
\end{equation*}
by $(\ref{lambda_P2})$ we have $\gamma_i>0$, thus the claim $(\ref{claim_P})$ is proved.

The $n\times n-$dimension diagonal matrix $\Omega_{\alpha}$ on Poincar\'e slice is~:
\begin{equation*}
\Omega_{\alpha}= \operatorname{diag} \{\lambda^{\alpha}-\lambda_1,\lambda^{\alpha}-\lambda_2,\cdots,\lambda^{\alpha}-\lambda_n\}.
\end{equation*}
Denote
\begin{equation}\label{gammap}
\gamma_p :=\max_{1\leq i\leq n}\gamma_i^{-1},
\end{equation}
by $(\ref{claim_P})$ we have~:
\begin{equation*}
|\Omega^{-1}_{\alpha}|\leq \max_{1\leq i\leq n}\{|\lambda^{\alpha}-\lambda_1|^{-1},|\lambda^{\alpha}-\lambda_2|^{-1},\cdots,|\lambda^{\alpha}-\lambda_n|^{-1}\} \leq \max_{1\leq i\leq n}\gamma_i^{-1} =\gamma_p .
\end{equation*}
According to the equation (\ref{inverse-d}), we suppose
\begin{equation}\label{e2}
\epsilon \gamma_p\leq \frac{1}{2}.
\end{equation}
Then for all $\alpha\in P$ we have :
\begin{equation}\label{OmegaP}
\begin{aligned}
|(\Omega_{\alpha}-\epsilon N)^{-1}|=\left|\sum_{s=0}^{n_d-1}\Omega_{\alpha}^{-s-1}(\epsilon N)^s\right| \leq
\sum_{s=0}^{n_d-1}|\Omega_{\alpha}^{-1}|^{s+1}\epsilon^s \leq 2\gamma_p(1-2^{-n_d}) :=\Gamma(\Lambda,\kappa),
\end{aligned}
\end{equation}
For a given $g_{\scriptscriptstyle P}\in H(B_r)^n$ on the Poincar\'e slice, by $(S,P)-$decomposition $(\ref{decomposition})$ and $(\ref{matrix_mod})$, using (\ref{OmegaP}) on $P$ we have~:
\begin{equation}\label{Est_Omega_P}
\begin{aligned}
\left|(\Omega-\epsilon N)^{-1}g_{\scriptscriptstyle P}\right|_r
&=\left|(\Omega-\epsilon N)^{-1}\sum_{\alpha\in P}g_{\alpha}z^{\alpha}\right|_r\leq \sum_{\alpha\in P}|g_{\alpha}|\left|(\Omega_{\alpha}-\epsilon N)^{-1}\right|r^{|\alpha|}\\
&\leq \Gamma(\Lambda,\kappa)\sum_{\alpha\in P}|g_{\alpha}|r^{|\alpha|}= \Gamma(\Lambda,\kappa) |g_{\scriptscriptstyle P}|_r.
\end{aligned}
\end{equation}
The equation (\ref{LonPoincare}) reads as~:
\begin{equation*}
\left(I+(\Omega-\epsilon N)^{-1}\mathcal{R}^{\epsilon}\right)\varphi_{\scriptscriptstyle P}= (\Omega-\epsilon N)^{-1}g_{\scriptscriptstyle P}.
\end{equation*}
By the estimates above and Lemma \ref{lemmaR}, if we choose $\epsilon$ sufficiently small such that
$|(\Omega-\epsilon N)^{-1}\mathcal{R}^{\epsilon}|<1$, then the operator $I+(\Omega-\epsilon N)^{-1}\mathcal{R}^{\epsilon}$ is invertible at a formal level and we have :
\begin{equation*}
\varphi_{\scriptscriptstyle P}= \left(I+(\Omega-\epsilon N)^{-1}\mathcal{R}^{\epsilon}\right)^{-1}(\Omega-\epsilon N)^{-1}g_{\scriptscriptstyle P},
\end{equation*}
thus the existence is proved.

Let us estimate the solution. We rewrite (\ref{LonPoincare}) as~:
\begin{equation*}
\varphi_{\scriptscriptstyle P}=(\Omega-\epsilon N)^{-1} g_{\scriptscriptstyle P}- (\Omega-\epsilon N)^{-1}
\mathcal{R}^{\epsilon}\varphi_{\scriptscriptstyle P},
\end{equation*}
taking the polydisc norm of the ball $B_r$ on both sides and by Lemma \ref{lemmaR} we obtain~:
\begin{equation*}
\begin{aligned}
\bm{|}\varphi_{\scriptscriptstyle P}\bm{|}_r
&=\left|(\Omega-\epsilon N)^{-1}g_{\scriptscriptstyle P}-(\Omega-\epsilon N)^{-1}\mathcal{R}^{\epsilon}\varphi_{\scriptscriptstyle P}\right|_r \\
&\leq \left|(\Omega-\epsilon N)^{-1}g_{\scriptscriptstyle P}\right|_r
+\left|(\Omega-\epsilon N)^{-1}\mathcal{R}^{\epsilon}\varphi_{\scriptscriptstyle P}\right|_r\\
&\leq \Gamma(\Lambda,\kappa) \bm{|}g_{\scriptscriptstyle P}\bm{|}_r
+ \Gamma(\Lambda,\kappa) \left|\mathcal{R}^{\epsilon}\varphi_{\scriptscriptstyle P}\right|_r\\
&\leq \Gamma(\Lambda,\kappa) \left|g_{\scriptscriptstyle P}\right|_r +
\epsilon \Gamma(\Lambda,\kappa)\tilde{C} \left|\varphi_{\scriptscriptstyle P}\right|_r,
\end{aligned}
\end{equation*}
where $\tilde{C}=\tilde{C}(n_d,\lambda_d)$ comes from (\ref{est_remainder}) which depends only on $\Lambda$.

To summarize, on the Poincar\'{e} slice, if $\epsilon$ is sufficiently small such that
\begin{equation}\label{e3}
\epsilon \Gamma(\Lambda,\kappa)\tilde{C}<1
\end{equation}
and (\ref{e2}) are satisfied, then there is a constant $C(\Lambda,\kappa)=\Gamma(\Lambda,\kappa)\tilde{C}$ which depends on $\Lambda$ and $\kappa$ such that~:
\begin{equation*}
\bm{|}\varphi_{\scriptscriptstyle P}\bm{|}_r \leq \frac{\Gamma(\Lambda,\kappa)}{1-\epsilon\cdot C(\Lambda,\kappa)} \bm{|}g_{\scriptscriptstyle P}\bm{|}_r .
\end{equation*}
If there is more than one Jordan block in $\Lambda^{\epsilon}$, then the estimate in $(\ref{OmegaP})$ will depend on the maximum of $n_d$ where $\mathcal{I}_d$ corresponds to Jordan blocks. Thus the lemma is proved.
\end{proof}
\begin{Remark}
We need to mention that this idea can not be applied on Siegel slice, because we will encounter small divisor there. In fact, on Siegel slice the estimate of $|(\Omega-\epsilon N)^{-1}|$ in $(\ref{OmegaP})$ is related to the multi-indices $\alpha$, which tends to infinity as $|\alpha|$ goes to infinity. So we can not choose $\epsilon$ sufficiently small to satisfy $(\ref{e3})$.
\end{Remark}
Before we start to estimate the solution of the homological equation restricted on the Siegel slice, we need to introduce the lexicographic order for monomials.
\begin{Definition}\label{def_lexico}
Given two exponent vectors $\alpha=(\alpha_1,\alpha_2,\cdots,\alpha_n),\beta=(\beta_1,\beta_2,\cdots,\beta_n)$. One has
$$(\alpha_1,\alpha_2,\cdots,\alpha_n)\prec(\beta_1,\beta_2,\cdots,\beta_n)$$
if either $|\alpha|<|\beta|,$ or $|\alpha|=|\beta|$ but $\alpha_i<\beta_i$ for the smallest $i$ for which $\alpha_i\neq\beta_i$. Then naturally we have a  relation order ``less than: $\prec$'' in $\mathbb{Z}^n_+(2)$. The {\it lexicographic order} of a monomial $z_1^{\alpha_1}\cdots z_n^{\alpha_n}$ is defined as~:
$$LO(z_1^{\alpha_1}\cdots z_n^{\alpha_n}) :=\alpha =(\alpha_1,\cdots,\alpha_n).$$
For any finite subset $W\subset \mathbb{Z}^n_+(2)$ and a power series $f_W(z)=\sum\limits_{\alpha\in W}f_{\alpha}z^{\alpha},$ define
\begin{equation*}
LO(f_W) :=\max_{\alpha\in W} LO(z^{\alpha}),
\end{equation*}
where the maximum is taken under the relation order $\prec$.
\end{Definition}
With this definition, we have the following lemma~:
\begin{Lemma}\label{nilR}
On Siegel slice $(\ref{decomposition})$, the operator $\mathcal{R}^{\epsilon}$ is nilpotent. i.e., there exists a finite integer $\eta=\eta(\kappa,n)$ which depends only on $\kappa$ and $n$ such that~:
\begin{equation*}
(\mathcal{R}^{\epsilon})^{\eta}g_{\scriptscriptstyle S}=0~~ ,~~ \text{for all}~~g_{\scriptscriptstyle S}\in H(B_r)^n.
\end{equation*}
\end{Lemma}
\begin{proof}
First we consider the case of a single Jordan block $\Lambda_d$. Recall that each Siegel slice is $(\ref{s_ik})$, on $S_1$ and $S_2$ we have~:
\begin{equation*}
g_{\scriptscriptstyle S_1}(z)= \sum_{|A_1|=2}^{\infty}g_{\scriptscriptstyle A_1}Z_1^{A_1}~ ,~ g_{\scriptscriptstyle S_2}(z)= \sum_{|A_1|=1}^{\infty}\sum_{i=2}^{m}\sum_{|A_i|=1}g_{\scriptscriptstyle A_1,A_i}Z_1^{A_1}Z_i^{A_i} .
\end{equation*}
By $(\ref{R_phi})$ and direct calculation we obtain :
\begin{equation}\label{R_S1}
\mathcal{R}^{\epsilon}(g_{\scriptscriptstyle S_1})=0
\end{equation}
For any $i\neq d$ we have $\mathcal{R}^{\epsilon}(Z_1^{A_1}Z_i^{A_i})=0$. Note that the operator $\mathcal{R}^{\epsilon}$ acts only on monomials, thus
\begin{equation*}
\mathcal{R}^{\epsilon}\left(g_{\scriptscriptstyle S_2}(z)\right)=\sum_{|A_1|=1}^{\infty}\sum_{|A_d|=1}g_{\scriptscriptstyle A_1,A_d}(\mathcal{R}^{\epsilon})(Z_1^{A_1}Z_d^{A_d}).
\end{equation*}
For each $j\in \mathcal{I}_d,j\neq n_d$ and $|A_j|=1$, by observing that $$\mathcal{R}^{\epsilon}(Z_1^{A_1}Z_j^{A_j})=\mathcal{R}^{\epsilon}(Z_1^{A_1}z_j)=\epsilon Z_1^{A_1}z_{j+1},$$
we have~:
\begin{equation}\label{R_S2}
(\mathcal{R}^{\epsilon})^{n_d}\left(g_{\scriptscriptstyle S_2}(z)\right)=\sum_{|A_1|=1}^{\infty}g_{\scriptscriptstyle A_1,A_{d_1}}\epsilon^{n_d-1}\mathcal{R}^{\epsilon}(Z_1^{A_1}z_{n_d})=0 .
\end{equation}

When there is a quasi-resonance $\kappa\in QR_i$, we have~:
\begin{equation}\label{g_S_ikappa}
g_{\scriptscriptstyle S_{i,\kappa}}(z)= \sum_{\alpha\in S_{i,\kappa}} g_{\scriptscriptstyle \alpha}Z_1^{A_1}Z_2^{A_2}\cdots Z_{i-1}^{A_{i-1}}.
\end{equation}
For each quasi$-$resonance, we consider the monomial $Z_1^{A_1}Z_2^{A_2}\cdots Z_{i-1}^{A_{i-1}}$ for a fixed $\alpha=(A_1,A_2,\cdots,A_{i-1},0)$ which belongs to the slice $S_{i,\kappa}$. First we observe that $\mathcal{R}^{\epsilon}(Z_1^{A_1}Z_2^{A_2}\cdots Z_{i-1}^{A_{i-1}})\neq 0$ iff $2\leq d\leq i-1$. Actually we have~:
\begin{align}
\nonumber&\mathcal{R}^{\epsilon}(Z_1^{A_1}Z_2^{A_2}\cdots Z_{i-1}^{A_{i-1}})\\
=\mathop{\prod}_{1\leq j\leq i-1 \atop j\neq d}(\Lambda_j Z_j)^{A_j}&\left[\prod_{i=1}^{n_d-1}(\lambda_{d_i} z_{d_i}+\epsilon z_{d_{i+1}})^{\alpha_{d_i}}\cdot(\lambda_{d_{n_d}} z_{d_{n_d}})^{\alpha_{d_{n_d}}}
-\prod_{i=1}^{n_d}(\lambda_{d_i} z_{d_i})^{\alpha_{d_i}}\right]\label{Rmono}.
\end{align}
We only expand terms of non zero exponent $\alpha_{d_i}\neq 0$ through binomial expansion. In order to simplify notations, in the following we suppress $d$ and substitute $n_d,d_i$ by $n,i$ respectively. For the first product in the bracket of (\ref{Rmono}), we distribute the first parentheses of the product as~:
\begin{align}
&\prod_{i=1}^{n-1}\left((\lambda_i z_i)^{\alpha_i}+\sum_{s_i=0}^{\alpha_i-1}\binom{\alpha_i}{s_i}(\lambda_i z_i)^{s_i}(\epsilon z_{i+1})^{\alpha_i-s_i}\right)\cdot(\lambda_n z_n)^{\alpha_n}\label{r1}\\
=&(\lambda_1 z_1)^{\alpha_1}\prod_{i=2}^{n-1}\left(\lambda_i z_i+\epsilon z_{i+1}\right)^{\alpha_i}\cdot(\lambda_n z_n)^{\alpha_n}+\label{a31}\\
&\sum_{s_1=0}^{\alpha_1-1}\binom{\alpha_1}{s_1}(\lambda_1 z_1)^{s_1}(\epsilon z_2)^{\alpha_1-s_1}\cdot\prod_{i=2}^{n-1}\left(\lambda_i z_i+\epsilon z_{i+1}\right)^{\alpha_i}\cdot(\lambda_n z_n)^{\alpha_n}\label{a32}.
\end{align}
We notice that all of the lexicographic orders of those terms in $(\ref{a32})$, by Definition \ref{def_lexico} are strictly less than $LO(z_1^{\alpha_1}z_2^{\alpha_2}\cdots z_n^{\alpha_n}),$ because the degree of $z_1$ decreases at least 1. Let us focus on the term $(\ref{a31}),$ we also expand the first parentheses of the product in (\ref{a31})~:
\begin{align}
\nonumber&(\lambda_1 z_1)^{\alpha_1}\prod_{i=2}^{n-1}\left(\lambda_i z_i+\epsilon z_{i+1}\right)^{\alpha_i}\cdot(\lambda_n z_n)^{\alpha_n}\\
\nonumber=&(\lambda_1 z_1)^{\alpha_1}\prod_{i=2}^{n-1}\left((\lambda_i z_i)^{\alpha_i}+\sum_{s_i=0}^{\alpha_i-1}\binom{\alpha_i}{s_i}(\lambda_i z_i)^{s_i}(\epsilon z_{i+1})^{\alpha_i-s_i}\right)\cdot(\lambda_n z_n)^{\alpha_n}\\
=& (\lambda_1 z_1)^{\alpha_1}(\lambda_2 z_2)^{\alpha_2}\prod_{i=3}^{n-1}\left(\lambda_i z_i+\epsilon z_{i+1}\right)^{\alpha_i}\cdot(\lambda_n z_n)^{\alpha_n}+\label{a41}\\
& (\lambda_1 z_1)^{\alpha_1}\sum_{s_2=0}^{\alpha_2-1}\binom{\alpha_2}{s_2}(\lambda_2 z_2)^{s_2}(\epsilon z_3)^{\alpha_2-s_2}\cdot\prod_{i=3}^{n-1}\left(\lambda_i z_i+\epsilon z_{i+1}\right)^{\alpha_i}\cdot(\lambda_n z_n)^{\alpha_n}\label{a42}.
\end{align}
Similarly we find that all of the lexicographic orders of those monomials in the term of (\ref{a42}), by Definition \ref{def_lexico} are strictly less than $LO(z_1^{\alpha_1}z_2^{\alpha_2}\cdots z_n^{\alpha_n}),$ because the degree of $z_2$ decreases at least 1 and the degree of $z_1$ is the same. As for term (\ref{a41}), we continue to expand the first parentheses in the product left. Thus when we expand all of the $n-1$ parentheses one by one, we can find that except from the first term
\begin{equation}\label{prod3-n}
\prod_{i=1}^{n}(\lambda_i z_i)^{\alpha_i},
\end{equation}
the lexicographic order of all the other terms produced in (\ref{r1}) decreases, for that there is always one $i$ in $\{1,2,\cdots,n-1\}$ such that the degree of $z_i$ decreases at least 1 and the degree of the previous $z_{j}$'s, $j\leq i-1$ keep the same. And the term (\ref{prod3-n}) of the original degree is eliminated in the bracket in (\ref{Rmono}). With these observations above, recall that we suppress $d$ before, by Definition \ref{def_lexico} we conclude that:
\begin{equation*}
LO\left(\mathcal{R}^{\epsilon}(Z_1^{A_1}Z_2^{A_2}\cdots Z_{i-1}^{A_{i-1}})\right)< LO(Z_1^{A_1}Z_2^{A_2}\cdots Z_{i-1}^{A_{i-1}}).
\end{equation*}
for a fixed $\alpha=(A_1,A_2,\cdots,A_{i-1},0)$. Since the formula (\ref{Rmono}) also tells that the new monomials produced by $\mathcal{R}^{\epsilon}$ still have the same indices of $A_j,j\neq d,$ the remainder operator only changes the exponents of the variables corresponding with $\Lambda_d$. Thus we apply the remainder operator $\mathcal{R}^{\epsilon}$ at most for a finite time, we will get the only one term which has the smallest lexicographic order~:
\begin{equation}\label{small_lo}
Z_1^{A_1}\cdots z_{n_d}^{\kappa^{(i)}_d}\cdots Z_{i-1}^{A_{i-1}},
\end{equation}
for an index $\alpha\in S_{i,\kappa}$ with fixed $A_j,j\neq d$. It is obvious that if we apply the remainder operator $\mathcal{R}^{\epsilon}$ to (\ref{small_lo}) once again, it turns into zero. Thus for each $A_d,$ there exists an integer $\eta_{A_d}$ such that
\begin{equation*}
(\mathcal{R}^{\epsilon})^{\eta_{A_d}}(Z_1^{A_1}Z_2^{A_2}\cdots Z_{i-1}^{A_{i-1}})=0.
\end{equation*}
While on each Siegel slice $S_{i,\kappa},$ we have a finite $A_d$ such that $|A_d|=\kappa^{(i)}_d,$ thus denote
\begin{equation*}
\eta_{i,\kappa} :=\max_{|A_d|=\kappa^{(i)}_d}\eta_{A_d},
\end{equation*}
and we have
\begin{equation}\label{R_S_ikappa}
(\mathcal{R}^{\epsilon})^{\eta_{i,\kappa}}g_{\scriptscriptstyle S_{i,\kappa}}=0.
\end{equation}

Since we have proved the nilpotency of the remainder operator $\mathcal{R}^{\epsilon}$ on each Siegel slice $S_{i,\kappa}$, note that the whole Siegel slice $S$ in $(\ref{decomposition})$ is a finite and disjoint union of $S_{i,\kappa},\kappa\in QR_i$. Considering $\sharp|QR_i|< \infty$ and
\begin{equation*}
g_{\scriptscriptstyle S}=\sum_{i=1}^{m}g_{\scriptscriptstyle S_i}=g_{\scriptscriptstyle S_1}+g_{\scriptscriptstyle S_2}+\sum_{i=3}^{m}\sum_{\kappa\in QR_i}g_{\scriptscriptstyle S_{i,\kappa}},
\end{equation*}
combining (\ref{R_S1}),(\ref{R_S2}) and (\ref{R_S_ikappa}), there exists a finite number
$$\eta(\kappa,n) :=\max\{1,n_d,\max_{\kappa\in QR_i \atop 3\leq i\leq m}\eta_{i,\kappa}\}$$
such that~:
\begin{equation*}
(\mathcal{R}^{\epsilon})^{\eta}g_{\scriptscriptstyle S}=0.
\end{equation*}
Thus we have proved our conclusion for only one single Jordan block. When there are several Jordan blocks, the monomial having the smallest lexicographic order in $(\ref{small_lo})$ will be the product of several monomials with the smallest lexicographic orders which corresponds to different blocks. Thus the lemma is proved.
\end{proof}

With this lemma above, we can solve and estimate the homological equation restricted on Siegel slice.
\begin{Lemma}\label{lemmaS}
The restriction to the Siegel slice of the homological equation $(\ref{homode})$ has a unique solution $\varphi_{\scriptscriptstyle S}$ w.r.t $g_{\scriptscriptstyle S}$ for any given $g_{\scriptscriptstyle S}\in H(B_r)^n$. Moreover, we have the following estimate~:
\begin{equation}\label{EstonS}
\left|\varphi_{\scriptscriptstyle S}\right|_{r-\delta}\leq C_S\cdot \delta^{-\vartheta} |g_{\scriptscriptstyle S}|_r,\quad \text{for all }~~0<\delta<r\leq1,
\end{equation}
where $\vartheta$ is a positive constant which depends only on $\kappa,\sigma,n$ and $C_S$ is a constant which depends only on $\kappa,\sigma,\Lambda,n,C_0$.
\end{Lemma}
\begin{proof}
First we consider the linear part having only one single Jordan block $\Lambda_d$. Without loss of generality, in order to make this proof simpler we will suppose that only the first eigenvalue belongs to the unit circle, i.e., $\mathcal{I}_1=\{1\}$. If $\mathcal{I}_1$ has more that one element, we give the similar assumption of Diophantine conditions and this makes no essential difference to our proof.

Similar to Poincar\'{e} slice case, by $(S,P)-$decomposition $(\ref{decomposition})$, on $S_1$ we have~:
$$\Omega_{\alpha}=\operatorname{diag}\{\lambda_1^{\alpha_1}-\lambda_1,\lambda_1^{\alpha_1}-\lambda_2,\cdots,\lambda_1^{\alpha_1}-\lambda_n\}.$$
By condition $(\ref{small_A1})$, we have
$$|\lambda_1^{\alpha_1}-\lambda_1|\geq C_0|\alpha_1|^{-\sigma},~\text{and}~|\lambda_1^{\alpha_1}-\lambda_i|\geq 1-|\lambda_i|~\text{for}~2\leq i\leq n.$$
Recall that $S_2=\{\alpha\in \mathbb{Z}_+^n(2): \sum_{l=2}^{m}|A_j|=1\}$, if $\alpha\in S_2$ then in the diagonal matrix
$$\Omega_{\alpha}=\operatorname{diag}\{\lambda^{\alpha}-\lambda_1,\lambda^{\alpha}-\lambda_2,\cdots,\lambda^{\alpha}-\lambda_n\},$$
$\alpha$ has and only has one $|A_j|=1,j\geq 2$, thus for $i,j\geq 2$ by $(\ref{small_A2})$ we have~:
\begin{align*}
|\lambda_1^{\alpha_1}\lambda_i-\lambda_j|&\geq |\lambda_j-\lambda_i|~\text{for}~|\lambda_i|\neq |\lambda_j|,\\
\left|\lambda_1^{\alpha_1}\lambda_i-\lambda_j\right|&\geq C_0(\alpha_1+1)^{-\sigma}~\text{for}~\lambda_i\neq\lambda_j~\text{but}~|\lambda_i|=|\lambda_j|,\\
|\lambda_1^{\alpha_1}\lambda_j-\lambda_i|&\geq |\lambda_j||\lambda_1^{\alpha_1}-1|\geq |\lambda_j|\cdot C_0|\alpha_1+1|^{-\sigma}~\text{for}~\lambda_i=\lambda_j~.
\end{align*}
Finally, for each $3\leq i\leq m$, we fix a $\kappa^{(i)}\in \mathbb{N}^{i-2}$.  Let $\alpha\in S_{i,\kappa}$, then by the $QR-$Diophantine condition of quasi-resonance $(\ref{def_qr_Dio})$  we have~:
\begin{equation*}
\left|\lambda_1^{\alpha_1}\cdot\prod_{j=2}^{i-1}\Lambda_j^{A_j}- \lambda_k\right|\geq C_0(|\alpha_1|+\sum_{j=2}^{i-1}\kappa^{(i)}_j)^{-\sigma},
\end{equation*}
for all $k\in \mathcal{I}_i,\sum_{j=1}^{i-1}|A_j|\geq 2$ and $|A_j|=\kappa^{(i)}_j,2\leq j\leq i-1$. Thus we can get the estimate $|\Omega_{\alpha}^{-1}|$ on each Siegel slice $S_{i,\kappa},$ which will be used below.

With the observations above, let us estimate $|(\Omega_{\alpha}-\epsilon N)^{-1}|$ when $\alpha\in S$. In order to simplify the following estimate, we denote $\omega(\alpha_1):=C_0^{-1}|\alpha_1|^{\sigma}$. We can assume $\omega(\alpha_1)\geq 1$ even if this means $|\alpha_1|$ is large enough. According to (\ref{inverse-d}), by our previous choice (\ref{e2}) there exists a constant $\gamma_s$ which depends on $\Lambda,\kappa$ such that when $\alpha\in S_1$ we have :
\begin{equation}\label{Omega_10}
|w_0^{-1}|:=|(\Omega_{ \alpha}-\epsilon N)^{-1}|=\left|\sum_{k=0}^{n-1}\Omega_{\alpha}^{-k-1}(\epsilon N)^k\right|\leq  \sum_{k=0}^{n-1}|\Omega_{ \alpha}^{-1}|^{k+1}\epsilon^k \leq \gamma_s n \cdot\omega(\alpha_1)^n.
\end{equation}
On the other hand, when $\alpha\in S_2$, we have:
\begin{equation}\label{Omega_11}
|w_1^{-1}|:=|(\Omega_{\alpha}-\epsilon N)^{-1}|=\left|\sum_{k=0}^{n-1}\Omega_{\alpha}^{-k-1}(\epsilon N)^k\right| \leq  \sum_{k=0}^{n-1}|\Omega_{ \alpha}^{-1}|^{k+1}\epsilon^k \leq \gamma_s n \cdot\omega(\alpha_1+1)^n.
\end{equation}
Although this means increasing $\gamma_s$, we can assume that $\gamma_s n\geq 1$ so that the right hand side above is greater that $1$.

For $3\leq i\leq m$, suppose there is quasi$-$resonance $\kappa\in QR_i$ with $\alpha\in S_{i,\kappa}$. Then, we have~:
\begin{equation}\label{Omega_1k}
|w_{\kappa}^{-1}|:=|(\Omega_{ \alpha}-\epsilon N)^{-1}|
=\left|\sum_{k=0}^{n-1}\Omega_{ \alpha}^{-k-1}(\epsilon N)^k\right|
\leq  \sum_{k=0}^{n-1}|\Omega_{ \alpha}^{-1}|^{k+1}\epsilon^k \leq \gamma_s n \cdot\omega(\alpha_1+\sum_{j=2}^{i-1}\kappa_j)^n.
\end{equation}
Similar to the Poincar\'e slice, the homological equation restricted on Siegel slice is~:
\begin{equation}\label{LonSiegel}
\mathcal{L}(\varphi_{\scriptscriptstyle S})
=(\Omega-\epsilon N)\varphi_{\scriptscriptstyle S}+ \mathcal{R}^{\epsilon}\varphi_{\scriptscriptstyle S}
= g_{\scriptscriptstyle S},\quad\text{for a given}\quad g_{\scriptscriptstyle S}\in H(B_r)^n.
\end{equation}
We write the formal solution of it~:
\begin{equation*}
\varphi_{\scriptscriptstyle S}= \left(I+(\Omega-\epsilon N)^{-1}\mathcal{R}^{\epsilon}\right)^{-1}(\Omega-\epsilon N)^{-1}g_{\scriptscriptstyle S}.
\end{equation*}
By $(\ref{Ldecomp})$ and $(\ref{OmegaN})$, for a given $g_{\scriptscriptstyle S}=\sum_{\alpha\in S}g_{\alpha}z^{\alpha}$, we have  $\Omega \mathcal{R}^{\epsilon}=\mathcal{R}^{\epsilon}\Omega$. Indeed, we have
\begin{eqnarray*}
\Omega \mathcal{R}^{\epsilon}g_{\scriptscriptstyle S} & = & g_{\scriptscriptstyle S}(\Lambda^2+\epsilon N  \Lambda)-g_{\scriptscriptstyle S}(\Lambda^2)-\Lambda\left( g_{\scriptscriptstyle S}(\Lambda+\epsilon N)-g_{\scriptscriptstyle S}(\Lambda)\right)\\
\mathcal{R}^{\epsilon}\Omega g_{\scriptscriptstyle S}&=& g_{\scriptscriptstyle S}(\Lambda^2+\epsilon N  \Lambda)-\Lambda g_{\scriptscriptstyle S}(\Lambda+\epsilon N)-g_{\scriptscriptstyle S}(\Lambda^2)-\Lambda g_{\scriptscriptstyle S}(\Lambda).
\end{eqnarray*}
On the other hand,
$\mathcal{R}^{\epsilon}N(\varphi)= N\varphi(\Lambda+\epsilon N)-N\varphi(\Lambda)=N\mathcal{R}^{\epsilon}(\varphi)$.
Hence, by observing $(\ref{inverse-d})$,  $(\Omega-\epsilon N)^{-1}$ and $\mathcal{R}^{\epsilon}$ are  pairwise commuting. Then according to Lemma \ref{nilR}, the inverse of the operator $I+(\Omega-\epsilon N)^{-1}\mathcal{R}^{\epsilon}$ does exist as a finite sum~:
\begin{equation*}
\left(I+(\Omega-\epsilon N)^{-1}\mathcal{R}^{\epsilon}\right)^{-1}=\sum_{j\geq0}(-1)^j\left((\Omega-\epsilon N)^{-1}\mathcal{R}^{\epsilon}\right)^j= \sum_{j=0}^{\eta-1}(-1)^j(\Omega-\epsilon N)^{-j}\left(\mathcal{R}^{\epsilon}\right)^j .
\end{equation*}
Thus we have proved the existence of the formal solution to the homological equation restricted on Siegel slice (\ref{LonSiegel}).

By Lemma \ref{invariance} we can estimate it on each Siegel slices $S_i$ separately. First on $S_1$, by equation $(\ref{Omega_10})$ we have:
\begin{align*}
|\varphi_{\scriptscriptstyle S_1}|_{r-\delta}
&=\left|\left(I+(\Omega-\epsilon N)^{-1}\mathcal{R}^{\epsilon}\right)^{-1}(\Omega-\epsilon N)^{-1}g_{\scriptscriptstyle S_1}\right|_{r-\delta}\\
&= \left|\sum_{\alpha_1=2}^{\infty}(\Omega_{\alpha}-\epsilon N)^{-1}g_{\scriptscriptstyle \alpha_1,0}z_1^{\alpha_1}\right|_{r-\delta}\\
&\leq \sum_{\alpha_1=2}^{\infty}|w_0|^{-1}|g_{\scriptscriptstyle \alpha_1,0}|(r-\delta)^{\alpha_1}\\
&\leq C_{S_1}\cdot \sum_{\alpha_1=2}^{\infty}\omega(\alpha_1)^n\left(\frac{r-\delta}{r}\right)^{\alpha_1}|g_{\scriptscriptstyle \alpha_1,0}|r^{\alpha_1} .
\end{align*}
And we have
\begin{equation}\label{deltaer}
1-\frac{\delta}{r}\leq 1-\delta\leq e^{-\delta}~\text{when}~0<\delta<r\leq 1.
\end{equation}
Let $h(x):=x^n e^{-\delta x},x\geq 2$, then $h'(x)=(n-\delta x)x^{n-1}e^{-\delta x}$, thus $h(x)$ achieves its maximum at $x_0=\frac{n}{\delta}$, i.e.,
\begin{equation}\label{x_n_e}
h(x)\leq h(\frac{n}{\delta})=\delta^{-n}\cdot n^ne^{-n}.
\end{equation}
Combining $(\ref{deltaer})$ and $(\ref{x_n_e})$ we obtain:
\begin{equation}\label{phi_S1}
|\varphi_{\scriptscriptstyle S_1}|_{r-\delta}
\leq C_{S_1}\sum_{\alpha_1=2}^{\infty}\alpha_1^{\sigma n}e^{-\delta \alpha_1}|g_{\scriptscriptstyle \alpha_1,0}|r^{\alpha_1}\leq C_{S_1}\cdot \delta^{-\sigma n}|g_{\scriptscriptstyle S_1}|_r .
\end{equation}
Note that among the estimates above we do not change the notation $C_{S_1}$, because it is constant which only depends on $\sigma,\Lambda,n,\kappa,C_0$ and it does not influence the convergence proof in section \ref{convergenceproof}.

On $S_2$, according to Lemma \ref{nilR} we have~:
\begin{align*}
|\varphi_{\scriptscriptstyle S_2}|_{r-\delta}
&=\left|\left(I+(\Omega-\epsilon N)^{-1}\mathcal{R}^{\epsilon}\right)^{-1}(\Omega-\epsilon N)^{-1}g_{\scriptscriptstyle S_2}\right|_{r-\delta}\\
&=\left|\sum_{j=0}^{n-1} \sum_{\alpha_1=1}^{\infty}\sum_{i=2}^{n}(-1)^j(\Omega_{\alpha}-\epsilon N)^{-(j+1)}g_{\scriptscriptstyle \alpha_1,e_i}\left(\mathcal{R}^{\epsilon}\right)^j \left(z_1^{\alpha_1}z_i\right) \right|_{r-\delta}.
\end{align*}
Since the sum for index $j$ is a finite sum, it will produce at most finite monomials of form $z_1^{\alpha_1}z_i$ for a fixed $\alpha_1$, we also assume $|(\Omega_{\alpha}-\epsilon N)|^{-1}\geq 1$, by $(\ref{Omega_11}),(\ref{deltaer})$ and Cauchy estimates we have~:
\begin{align*}
|\varphi_{\scriptscriptstyle S_2}|_{r-\delta}
&\leq C_{S_2}\cdot\sum_{\alpha_1=1}^{\infty}\sum_{i=2}^{n}|w_1|^{-n}|g_{\scriptscriptstyle \alpha_1,e_i}|(r-\delta)^{\alpha_1+1}\\
&\leq C_{S_2}\sum_{\alpha_1=1}^{\infty} \omega(\alpha_1+1)^{n^2} \left(\frac{r-\delta}{r}\right)^{\alpha_1+1}|g_{\scriptscriptstyle \alpha_1,e_i}|r^{\alpha_1+1},
\end{align*}
where $C_{S_2}$ is a constant which only depends on $\sigma,\Lambda,n,\kappa,C_0$. Thus by $(\ref{x_n_e})$ we obtain~:
\begin{equation}\label{phi_S2}
|\varphi_{\scriptscriptstyle S_2}|_{r-\delta}
\leq C_{S_2}\cdot \sum_{\alpha_1=2}^{\infty}\alpha_1^{\sigma n^2}e^{-\delta \alpha_1}|g_{\scriptscriptstyle \alpha_1,e_i}|r^{\alpha_1+1}\leq C_{S_2}\cdot \delta^{-\sigma n^2}|g_{\scriptscriptstyle S_2}|_r .
\end{equation}

Finally, let us consider the slice $S_{i,\kappa}$. For a fixed $\kappa\in QR_i,$ by Lemma \ref{nilR} and (\ref{deltaer}) we have~:
\begin{align*}
|\varphi_{\scriptscriptstyle S_{i,\kappa}}|_{r-\delta}
&=\left|\left(I+(\Omega-\epsilon N)^{-1}\mathcal{R}^{\epsilon}\right)^{-1}(\Omega-\epsilon N)^{-1}g_{\scriptscriptstyle S_i}\right|_{r-\delta}\\
&=\left| \sum_{j=0}^{\eta-1}(-1)^j(\Omega-\epsilon N)^{-(j+1)}(\mathcal{R}^{\epsilon})^j\sum_{\alpha_1=0}^{\infty} \mathop{\sum}_{2\leq j\leq i-1 \atop |A_j|=\kappa_j}g_{\scriptscriptstyle \alpha}z_1^{\alpha_1}Z_2^{A_2}\cdots Z_{i-1}^{A_{i-1}} \right|_{r-\delta}\\
&=\left|\sum_{j=0}^{\eta-1}\sum_{\alpha_1=0}^{\infty}\mathop{\sum}_{2\leq j\leq i-1 \atop |A_j|=\kappa_j}(\Omega_{\alpha}-\epsilon N)^{-(j+1)}g_{\scriptscriptstyle \alpha}\left(\mathcal{R}^{\epsilon}\right)^j\left(z_1^{\alpha_1}Z_2^{A_2}\cdots Z_{i-1}^{A_{i-1}}\right) \right|_{r-\delta}.
\end{align*}
Similar to the estimate on $S_2$, since the sum for index $j$ is finite and is related to $\eta$, it will produce at most finite monomials of form $z_1^{\alpha_1}Z_2^{A_2}\cdots Z_{i-1}^{A_{i-1}}$ for a fixed $\alpha_1$, we can also assume $|(\Omega_{\alpha}-\epsilon N)|^{-1}\geq 1$, by  $(\ref{Omega_1k}),(\ref{deltaer})$ and $(\ref{x_n_e})$ we have:
\begin{equation}\label{phi_Sik}
\begin{aligned}
|\varphi_{\scriptscriptstyle S_{i,\kappa}}|_{r-\delta}
&\leq C_{S_{i,\kappa}}\sum_{\alpha_1=0}^{\infty}\mathop{\sum}_{2\leq j\leq i-1 \atop |A_j|=\kappa_j} \left|\omega\left(\alpha_1+\sum_{j=2}^{i-1}\kappa_j\right)^{\eta n}\right||g_{\scriptscriptstyle \alpha}|(r-\delta)^{\alpha_1+\sum_{j=2}^{i-1}\kappa_j}\\
&\leq C_{S_{i,\kappa}}\cdot\sum_{\alpha_1=0}^{\infty}\left(\alpha_1+\sum_{j=2}^{i-1}\kappa_j\right)^{\sigma\eta n}\left(\frac{r-\delta}{r}\right)^{\alpha_1+\sum_{j=2}^{i-1}\kappa_j}|g_{\scriptscriptstyle \alpha}|r^{\alpha_1+\sum_{j=2}^{i-1}\kappa_j}\\
&\leq C_{S_{i,\kappa}}\cdot\delta^{-\sigma\eta n}|g_{\scriptscriptstyle S_{i,\kappa}}|_r.
\end{aligned}
\end{equation}

With the above estimates on these Siegel slices, combining $(\ref{phi_S1}),(\ref{phi_S2})$ and $(\ref{phi_Sik})$, by the invariance Lemma \ref{invariance} we have~:
\begin{align*}
\left|\varphi_{\scriptscriptstyle S}\right|_{r-\delta}
&=\left|\sum_{i=1}^{m}\varphi_{\scriptscriptstyle S_i}\right|_{r-\delta}=\left|\varphi_{\scriptscriptstyle S_1}\right|_{r-\delta}+\left|\varphi_{\scriptscriptstyle S_2}\right|_{r-\delta}
+\left|\sum_{i=3}^{m}\sum_{\kappa\in QR_i}\varphi_{\scriptscriptstyle S_{i,\kappa}}\right|_{r-\delta}\\
&\leq C_{S_1}\cdot \delta^{-\sigma n}|g_{\scriptscriptstyle S_1}|_r + C_{S_2}\cdot\delta^{-\sigma n^2}|g_{\scriptscriptstyle S_2}|_r+\sum_{i=3}^{m}\sum_{\kappa\in QR_i}C_{S_{i,\kappa}}\cdot\delta^{-\sigma\eta n}|g_{\scriptscriptstyle S_{i,\kappa}}|_r\\
&\leq C_S\cdot \delta^{-\vartheta} |g_{\scriptscriptstyle S}|_r,
\end{align*}
where $C_S$ depends only on $\kappa,\sigma,\Lambda,n,C_0$ and $\vartheta=\max \{\sigma n,\sigma n^2,\sigma\eta n\}$, which depends only on $\kappa,\eta,\sigma,n$. Thus the lemma is proved.
\end{proof}

With those lemmas above, eventually we obtain the estimate of the solution to the homological equation (\ref{Lphi_g}).
\begin{proof}[Proof of Proposition \ref{ProJn}]
Suppose $\epsilon$ sufficiently small so that it satisfies (\ref{e2}) and (\ref{e3}), i.e.,
\begin{equation}\label{epsilon}
\epsilon< \min\{(2\gamma_p)^{-1},(n\Gamma(\Lambda,\kappa)\tilde{C})^{-1}\}.
\end{equation}
Combining (\ref{polyineq}),(\ref{PE}) and (\ref{deltaer}), from Lemma \ref{lemmaP} and \ref{lemmaS} we conclude that~:
\begin{align*}
\|\varphi\|_{r-\delta}
&\leq |\varphi|_{r-\delta}=|\varphi_{\scriptscriptstyle P}|_{r-\delta}+|\varphi_{\scriptscriptstyle S}|_{r-\delta}\leq C_P\cdot |g_{\scriptscriptstyle P}|_{r-\delta}+ C_S\cdot \left(\frac{\delta}{2}\right)^{-\vartheta}|g_{\scriptscriptstyle S}|_{r-\frac{\delta}{2}}\\
&\leq C_2 \cdot \left(\frac{\delta}{2}\right)^{-\vartheta}\left(|g_{\scriptscriptstyle P}|_{r-\frac{\delta}{2}}+|g_{\scriptscriptstyle S}|_{r-\frac{\delta}{2}}\right)= C_2 \cdot \left(\frac{\delta}{2}\right)^{-\vartheta}|g|_{r-\frac{\delta}{2}}\\
&\leq C_2 \cdot \left(\frac{\delta}{2}\right)^{-\vartheta}\cdot \left(\frac{2r}{\delta}\right)^n\|g\|_r\leq C_2 \cdot \delta^{-(\vartheta+n)}\|g\|_r.
\end{align*}
here we use $r\leq 1$ and we also keep the notation of the constant $C_2$, which depends only on $\kappa,\sigma,\Lambda,n,C_0$.
Thus the proposition is proved.
\end{proof}

\section{Convergence proof}\label{convergenceproof}

With all of the preparations above,  in order to proof the holomorphic linearization, we can use verbatim the proof by the Newton iteration method as introduced by E.Zehnder \cite{Zeh77}. We shall recall its main points without proof.

We first give the most important lemma in our convergence proof. The main idea comes from E.Zehnder\cite{Zeh77}:
\begin{Lemma}\label{EofHomo}
Let $\epsilon$ satisfies $(\ref{epsilon})$, $g\in \mathcal{H}(B_r)^n$ for some $0<r\leq 1$, then the homological equation $(\ref{Lphi_g})$ has a unique solution $\varphi\in {\mathcal{H}(D_r)}^n$, moreover, the following estimate holds for all $0<\delta<r\leq1:$
\begin{equation}\label{Ephiofg}
\|\varphi\|_{C^1,r-\delta}\leq C_3 \delta^{-\tau}\|g\|_r.
\end{equation}
where $\tau=\vartheta+n+1$, which depends only on $\kappa,\sigma,n$ and $C_3$ is a constant depending on $\kappa,\sigma,C_0,|\Lambda^{\epsilon}|$ and $|(\Lambda^{\epsilon})^{-1}|$.
\end{Lemma}
\begin{proof}
Since $\varphi$ is holomorphic on $B_{r-\delta}$ and solves there the holomological equation (\ref{homode}), then $\varphi\circ \Lambda^{\epsilon}$ is also holomorphic there. 
According to the estimate (\ref{Est_L}), we have
\begin{equation*}
\|\varphi\circ\Lambda^{\epsilon}\|_{r-\delta}\leq |\Lambda^{\epsilon}|\|\varphi\|_{r-\delta}+\|g\|_{r-\delta}
\leq(|\Lambda^{\epsilon}|C_2+1)\cdot\delta^{-(\vartheta+n)}\|g\|_r,
\end{equation*}
since $\delta<1$. And by Cauchy estimate we have
\begin{equation*}
\begin{aligned}
\|D\varphi\|_{r-\delta}&\leq 2\delta^{-1}\|\varphi\|_{r-\frac{\delta}{2}}\leq 2^{\vartheta+n+1}C_2 \cdot\delta^{-(\vartheta+n+1)}\|g\|_r ,\\
\|D\varphi\circ \Lambda^{\epsilon}\|_{r-\delta}&\leq \|D(\varphi\circ \Lambda^{\epsilon})\|_{r-\delta}|(\Lambda^{\epsilon})^{-1}|
\leq 2\delta^{-1}|(\Lambda^{\epsilon})^{-1}|\|\varphi\circ \Lambda^{\epsilon}\|_{r-\frac{\delta}{2}}\\
&\leq |(\Lambda^{\epsilon})^{-1}|(|\Lambda^{\epsilon}|C_2+1)2^{\vartheta+n+1}\cdot\delta^{-(\vartheta+n+1)}\|g\|_r,
\end{aligned}
\end{equation*}
where $C_2$ comes from the estimate (\ref{Est_L}) which depends only on $\kappa,\sigma,\Lambda,n,C_0$. Thus we can choose a suitable constant $C_3$ which is much greater than the coefficients above, which concludes the proof.
\end{proof}

\subsection{Idea}
We consider the linearization problem with $\Phi=Id+\phi_{\geq 2}$, where $\phi_{\geq 2}\in {\mathcal{H}(B_r)}^n$. We try to solve the equation $\mathcal{F}(\Phi)=0$, where
\begin{equation*}
\mathcal{F}(\Phi):=F\circ \Phi-\Phi\circ\Lambda^{\epsilon}.
\end{equation*}
Since $\mathcal{F}(Id)=f$, which is small near $0$ , we are dealing with a perturbation problem. Since $f$ contains only terms of order $\geq 2$, then by eventually conjugating by an homothety, we can assume without loss of generality, that $f$ is holomorphic on $|z|<1$, and that
\begin{equation}\label{perturf}
\|f\|_1< \delta_0,
\end{equation}
for $\delta_0$ as small as we want, to be chosen later on. 
Assuming $\mathcal{F}(\Phi)$ to be small, we are looking for a better approximation $\Phi+v$, which makes $\mathcal{F}(\Phi+v)$ smaller. By Taylor expansion on Banach space we have~:
\begin{equation*}
\mathcal{F}(\Phi+v) = \mathcal{F}(\Phi)+\mathcal{F}'(\Phi)v+\mathrm{R}(\Phi,v),
\end{equation*}
where
\begin{equation}\label{Fuv}
\mathcal{F}'(\Phi)v:=\frac{d}{dt}\mathcal{F}(\Phi+tv)\left|_{t=0}\right.= DF\circ \Phi\cdot v-v\circ \Lambda^{\epsilon},
\end{equation}
and the high order term $\mathrm{R}(\Phi,v)$ is given by
\begin{equation}\label{remainder}
\mathrm{R}(\Phi,v)=\int_{0}^{1}(1-t)\frac{d^2}{dt^2}\mathcal{F}(\Phi+tv)\cdot v^2dt=\int_{0}^{1}(1-t)\frac{d^2}{dt^2}f(\Phi+tv)\cdot v^2dt.
\end{equation}
We would have to solve $\mathcal{F}(\Phi)+\mathcal{F}'(\Phi)v=0$ such that $\mathcal{F}(\Phi+v)=O_2(\mathcal{F}(\Phi))$. Unfortunately, because of small divisors the linear operator $\mathcal{F}'(\Phi)$ given by (\ref{Fuv}) has no right-inverse on the space of holomorphic map on a fixed domain. We need to construct a sufficiently good approximating right-inverse of $\mathcal{F}'(\Phi)$. Following R\"ussmann \cite{Ru72} we have 
\begin{equation}\label{CRFu}
D\mathcal{F}(\Phi)(z)=DF\circ \Phi(z)\cdot D\Phi(z)- D\Phi\circ\Lambda^{\epsilon}(z)\cdot \Lambda^{\epsilon}.
\end{equation}
Let us set :
\begin{equation*}
v:= D\Phi\cdot\varphi.
\end{equation*}
Combining (\ref{Fuv}) and (\ref{CRFu}) we obtain
\begin{equation*}
\mathcal{F}'(\Phi)v=D\mathcal{F}(\Phi)\cdot\varphi+D\Phi\circ\Lambda^{\epsilon}(\Lambda^{\epsilon}\varphi-\varphi\circ\Lambda^{\epsilon}).
\end{equation*}
Consequently
\begin{equation}\label{FuvRdF}
\mathcal{F}(\Phi+v)=\mathcal{F}(\Phi)+D\Phi\circ \Lambda^{\epsilon}(\Lambda^{\epsilon}\varphi-\varphi\circ\Lambda^{\epsilon})+\mathrm{R}(\Phi,v)+D\mathcal{F}(\Phi)\cdot\varphi.
\end{equation}
By Lemma \ref{EofHomo}, 
we are able to solve the equation
$\mathcal{F}(\Phi)+D\Phi\circ \Lambda^{\epsilon}(\Lambda^{\epsilon}\varphi-\varphi\circ\Lambda^{\epsilon})=0$ in case that $D\Phi\circ\Lambda^{\epsilon}$ is invertible. According to (\ref{FuvRdF}) we still have $\mathcal{F}(\Phi+v)=O_2(\mathcal{F}(\Phi))$.

\subsection{Set up}
With the linear operator defined in $(\ref{Lphi_g})$, we shall define inductively the iteration $\Phi_{\nu}$, for $\nu=0,1,2,\cdots$ as follows:
$\Phi_0=Id$, and for $\nu\geq 0$, set
\begin{equation}\label{indfor}
\begin{aligned}
\Phi_{\nu+1}&=\Phi_{\nu}+v_{\nu} ,\\
v_{\nu} &= D\Phi_{\nu}\cdot\varphi_{\nu} ,\\
\varphi_{\nu}&=\mathcal{L}^{-1}((D\Phi_{\nu}\circ\Lambda^{\epsilon})^{-1}\mathcal{F}(\Phi_{\nu})).
\end{aligned}
\end{equation}
With this formula and (\ref{FuvRdF}) we then have :
\begin{equation}\label{indF}
\mathcal{F}(\Phi_{\nu+1})= D\mathcal{F}(\Phi_{\nu})\cdot\varphi_{\nu}+\mathrm{R}(\Phi_{\nu},v_{\nu}),
\end{equation}
which is $O_2(\mathcal{F}(\Phi_{\nu}))$.
The domains $B_{r_{\nu}}$ are defined with
\begin{equation*}
r_{\nu}=\frac{1}{2}(1+2^{-(\nu+1)}), \quad \nu\geq 0.
\end{equation*}
Clearly $\lim\limits_{\nu\rightarrow\infty}r_{\nu}=\frac{1}{2}$, and $B_{r_{\nu+1}}\subset B_{r_{\nu}}$ for all $\nu\geq 0$. The sequence of small numbers, $\varepsilon_{\nu}$, is defined by
\begin{equation*}
\varepsilon_{\nu+1}=C^{\nu+1}\varepsilon^2_{\nu}, \quad \nu\geq 0,
\end{equation*}
where $C$ is a large constant depending on $\kappa,\sigma, C_0,n,|\Lambda^{\epsilon}|,|(\Lambda^{\epsilon})^{-1}|$ which will be determined later on. For
$\varepsilon_0$ sufficiently small, the sequence $\varepsilon_{\nu}$ tends rapidly to $0$, actually for all $\nu\geq 0$, (refer to the explicit calculation in\cite{dela99}) we have:
\begin{equation}\label{epsiloneq}
\varepsilon_{\nu}=C^{-(\nu+2)}(C^2\varepsilon_0)^{2^{\nu}}.
\end{equation}
In particularly we then have
\begin{equation}\label{epsilonin}
\varepsilon_{\nu+1}\leq \frac{1}{2}\varepsilon_{\nu}\leq\varepsilon_{\nu}-\varepsilon_{\nu+1}.
\end{equation}
With all of these preparations above, we can get the induction lemma.

\subsection{Induction}
We shall prove that if $\varepsilon_0$ is sufficiently small (that is to say $C$ is sufficiently large), then the following inductive lemma holds true for all $\nu\geq 0$, defined by (\ref{indfor}) inductively. It follows verbatim from E.Zehnder\cite{Zeh77}:
\begin{Lemma}\cite{Zeh77}
For all $\nu=0,1,2,\cdots$. we have :\\
$(1.\nu)$  $\Phi_{\nu}$ is holomorphic on $D_{r_{\nu}}, \Phi_{\nu}(0)=0, D\Phi_{\nu}(0)=1$, and
\begin{equation*}
\|\Phi_{\nu}-id\|_{C^1,r_{\nu}}\leq \varepsilon_0-\varepsilon_{\nu}.
\end{equation*}
$(2.\nu)$  $\mathcal{F}(\Phi_{\nu})$ is holomorphic on $B_{r_{\nu}}, $ and
\begin{equation*}
\|\mathcal{F}(\Phi_{\nu})\|_{r_{\nu}}\leq \varepsilon^2_{\nu}.
\end{equation*}
$(3.\nu)$  $v_{\nu}$ is holomorphic on $D_{r_{\nu+1}}, v_{\nu}(0)=0, Dv_{\nu}(0)=0$, and
\begin{equation*}\label{}
\|v_{\nu}\|_{C^1,r_{\nu+1}}\leq \varepsilon_{\nu+1}\leq \varepsilon_{\nu}-\varepsilon_{\nu+1}.
\end{equation*}
\end{Lemma}

From $(3.\nu), \nu=0,1,2,\cdots$ we can conclude, by (\ref{indfor}) $v_{\nu}=\Phi_{\nu+1}-\Phi_{\nu}$ that
$\Phi(z):= \lim\limits_{\nu \rightarrow \infty}= Id+\sum\limits_{k=0}^{\nu-1}v_k=\lim\limits_{\nu\rightarrow \infty}\Phi_{\nu}$ converges uniformly for $z\in B_{\frac{1}{2}}$. Hence $\Phi$ is a holomorphic map defined on $B_{\frac{1}{2}}$. From $(1.\nu)$ we have $\Phi(0)=0, D\Phi(0)=1$. As a consequence of $(2.\nu)$ we have on $B_{\frac{1}{2}}, \mathcal{F}(\Phi)=\lim\limits_{\nu\rightarrow\infty}\mathcal{F}(\Phi_{\nu})=0,$ which proves our main Theorem \ref{mainth1}.

\normalem
\bibliographystyle{alpha}

\end{document}